\documentclass[11pt]{article}
\usepackage{amsfonts,amsmath,latexsym,verbatim,amscd,mathrsfs,color,array}

\usepackage{amsmath,amssymb,amsthm,amsfonts,graphicx,color}
\usepackage{amssymb}
\usepackage{pdfsync}
\usepackage{epstopdf}

\usepackage[english]{babel}
\usepackage{amscd}
\usepackage{verbatim}
\usepackage{float}

\begin{document}
\newcommand{\ana}{\nabla}
\newcommand{\R}{I \!  \! R}
\newcommand{\N}{I \!  \! N}
\newcommand{\Ren}{ I \! \! R^n}


\newcommand{\be}{\begin{equation}}
\newcommand{\ee}{\end{equation}}
\newcommand{\bea}{\begin{eqnarray}}
\newcommand{\eea}{\end{eqnarray}}
\newcommand{\bean}{\begin{eqnarray*}}
\newcommand{\eean}{\end{eqnarray*}}
\newcommand{\la}{\label}
\newcommand{\bx}{\bar{x}}


\newcommand{\pp}{\partial}
\newcommand{\inn}{\quad\hbox{in }}
\newcommand{\xa}{\alpha}
\newcommand{\xb}{\beta}
\newcommand{\xg}{\gamma}
\newcommand{\RR}{\mathbb{R}}
\renewcommand{\R}{\mathbb{R}}
\newcommand{\xG}{\Gamma}
\newcommand{\foral}{\quad\hbox{for all }} 
\newcommand{\ttt}{\tilde}
\newcommand{\xd}{\delta}
\newcommand{\xD}{\Delta}
\newcommand{\xe}{\varepsilon}
\newcommand{\xz}{\zeta}
\newcommand{\xh}{\eta}
\newcommand{\Th}{\Theta}
\newcommand{\xk}{\kappa}
\newcommand{\xl}{\lambda}
\newcommand{\xL}{\Lambda}
\newcommand{\CC}{\mathcal{C}}
\newcommand{\ve}{\varepsilon}
\newcommand{\xm}{\mu}
\newcommand{\xn}{\nu}
\newcommand{\ks}{\rho}
\newcommand{\KS}{\Xi}
\newcommand{\xp}{\pi}
\newcommand{\xP}{\Pi}
\newcommand{\xr}{\rho}
\newcommand{\xs}{\sigma}
\newcommand{\xS}{\Sigma}
\newcommand{\xf}{\phi}
\newcommand{\xF}{\Phi}
\newcommand{\ps}{\psi}
\newcommand{\PS}{\Psi}
\newcommand{\xo}{\omega}
\newcommand{\xO}{\Omega}
\newcommand{\tA}{\tilde{A}}
\newcommand{\tC}{\bar{k}_s}
\newcommand{\qcap}{C_{\frac{2}{q},\;q'}}
\newcommand{\qsub}{\subset^q}
\newcommand{\qeq}{\sim^q}
\newcommand{\finedim}{{\hfill $\Box$}}


\newtheorem{theorem}{Theorem}[section]
\newtheorem{lemma}[theorem]{Lemma}
\newtheorem{prop}[theorem]{Proposition}
\newtheorem{coro}[theorem]{Corollary}
\newtheorem{defin}[theorem]{Definition}
\newtheorem{remark}[theorem]{Remark}
\newtheorem{notation}[theorem]{Notation}
\newcounter{newsec} \renewcommand{\theequation}{\thesection.\arabic{equation}}

\title{Ancient shrinking spherical interfaces in the Allen-Cahn flow }

\author{{\bf \large Manuel del Pino}\\
{\small  Departamento de
Ingenier\'{\i}a  Matem\'atica}\\
{\small and Centro de Modelamiento
 Matem\'atico (UMI 2807 CNRS) }\\
{\small  Universidad de Chile,
Casilla 170 Correo 3, Santiago,
Chile.}\\
{\it\small email: delpino@dim.uchile.cl}\\[0.5cm]
{\bf\large   Konstantinos T. Gkikas}\\
 \small Centro de Modelamiento Matem\'atico (UMI 2807 CNRS),\\
 \small Universidad de Chile,\\
\small Casilla 170 Correo 3, Santiago,
Chile.\\
{\emph{\small email: kgkikas@dim.uchile.cl}}}

\maketitle

\begin{abstract}
 We consider the parabolic Allen-Cahn equation in $\R^n$, $n\ge 2$,
$$u_t=  \Delta u + (1-u^2)u \quad \hbox{ in } \R^n \times (-\infty, 0].$$
We construct an ancient  radially symmetric solution $u(x,t)$  with any given number $k$ of transition layers between $-1$ and $+1$. At main order they consist of $k$ time-traveling copies of $w$ with spherical  interfaces distant $O(\log |t|  )$ one to each other as $t\to -\infty$.  These interfaces
are resemble at main order copies of the {\em shrinking sphere} ancient solution to mean the flow by mean curvature of surfaces:  $|x| = \sqrt{- 2(n-1)t}$.
More precisely, if $w(s)$ denotes the heteroclinic 1-dimensional solution of $w'' + (1-w^2)w=0$  $w(\pm \infty)= \pm 1$ given by
$w(s) = \tanh \left(\frac s{\sqrt{2}} \right) $ we have
$$
u(x,t) \approx   \sum_{j=1}^k (-1)^{j-1}w(|x|-\rho_j(t))  - \frac 12 (1+ (-1)^{k}) \quad \hbox{ as } t\to -\infty
$$
where
$$\rho_j(t)=\sqrt{-2(n-1)t}+\frac{1}{\sqrt{2}}\left(j-\frac{k+1}{2}\right)\log\left(\frac {|t|}{\log |t| }\right)+ O(1),\quad j=1,\ldots ,k.$$


\end{abstract}

\date{}


\setcounter{equation}{0}
\section{Introduction}
A classical model for phase transitions is the Allen-Cahn equation \cite{ac}
\be u_t = \Delta u + f(u) \inn \R^n \times \R  \label{ac0} \ee
where $f(u) = - F'(u)$ where $F$ is a {\em balanced bi-stable potential}
namely $F$ has exactly two non-degenerate global minimum points $u=+1$ and $u=-1$. The model is
\be \label{F} F(u) = - \frac 14(1-u^2)^2, \quad f(u) = (1-u^2)u . \ee
The constant functions  $u=\pm 1$ correspond to stable equilibria of Equation \eqref{ac0}.  They are idealized as two phases of a
material. A solution $u(x,t)$ whose values lie at all times in $[-1,1]$ and in most of the space $\R^n$ takes values close to either $+1$ or $-1$ corresponds to a continuous realization of the phase state of the material, in which the two stable states coexist.

\medskip
There is a broad literature on this type of solutions (in the static and dynamic cases). The main point is to derive qualitative
information on the ``interface region'', that is the walls separating the two phases. A close connection between these walls and minimal surfaces and surfaces evolving by mean curvature has been established in many works. To explain this connection,
it is convenient to introduce a small parameter $\ve$ and consider the scaled version of
\eqref{ac0}  for $u^\ve (x,t) = u(\ve x, \ve^2t)$,
\be u^\ve_t = \Delta u^\ve  + \varepsilon^{-2} f(u^\ve ) . \label{ac1} \ee
Let us consider a smooth embedded, orientable  hypersurface $\Sigma_0$ that separates $\R^n \setminus \Sigma_0$ into two components  $\Lambda_0^-$ and $\Lambda_0^+$ and
the characteristic function
$$
u_{\Sigma_0} (x) = \left \{ \begin{matrix}- 1  &\hbox{ if }  x\in \Lambda_0^-\\  +  1  &\hbox{ if }  x\in \Lambda_0^+\\    \end{matrix} \right .\ .
$$
The following principle  (in suitable senses) has been explored in a number of works: the solution $u^\ve(x,t)$ of equation \eqref{ac1} with initial condition $u^\ve(x,0) $ given by a suitable $\ve$-regularization of $u_{\Sigma_0} (x)$ satisfies
\be\label{limit}
\lim_{\ve \to 0 } u^\ve(x,t) \ =\ u_{\Sigma(t)} (x) , \quad t> 0,
\ee
where the surfaces $\Sigma(t)$ in $\R^n$ {\em evolve by mean curvature}. In the smooth case this means that
 each point of $\Sigma(t)$ moves in the normal direction
with a velocity proportional to its mean curvature at that point.
More precisely, there is a smooth family of  diffeomorphisms  $Y(\cdot, t) :\Sigma_0 \to \Sigma(t)$, $t>0$ with $\quad Y(y,0) = y$,
determined by the mean curvature flow equation

\be
\frac {\pp Y}{\pp t}   = H_{\Sigma(t)} (Y) \nu ( Y),
\label{mcf}\ee
where
$H_{\Sigma(t)} (Y) $ 
  designates the mean curvature of the surface $\Sigma(t)$ at the point $Y(y, t)$, $y\in \Sigma_0$, namely the trace of its second fundamental form, $\nu$ is a choice of unit normal vector that points towards $\Lambda_+$ at $t=0^+$.
Besides \eqref{limit}, the profile of $u^\ve$ near the surface $\Sigma(t)$ is given by
\be
u^\ve(x,t )  \approx  w \left( \frac s\ve  \right ) , \quad x= Y + s \nu(Y),
\label{limit2}\ee
where $w(s)$ is the unique (heteroclinic) solution to
$$
w'' + f(w) = 0 \inn \R, \quad w(0)=0,  \quad w(\pm \infty) =\pm 1
$$
which exists and it is monotone. In the special case \eqref{F}, it is given by $$ w(s) = \tanh\left ( \frac s{\sqrt{2}} \right). $$
These asymptotic laws were first suggested by Allen-Cahn \cite{ac}, then formally derived by Rubinstein-Sternberg-Keller \cite{rsk} and
de Mottoni-Schatzmann \cite{ds}. Rigorous results on this line were obtained in the radial case by Bronsard-Kohn \cite{bk}, and more in general by X. Chen \cite{chen}.  In \cite{ilmanen}, Ilmanen proved the convergence (in a measure theoretical sense) to Brakke's motion by mean curvature, for a setting not necessarily regular. S\'aez \cite{saez} investigated the (smooth) link in $\R^2$ with curve-shortening flow.

In the static case, the connection between interfaces and minimal surfaces $\Sigma$, namely  $H_\Sigma =0$, has been investigated in many works starting with Modica \cite{modica},
giving rise in particular to De Giorgi's conjecture on the connection of the elliptic Allen-Cahn equation with Bernstein's problem \cite{dg}.
 See for instance \cite{dkw1,dkw2,dkw3,ks,pacardritore,tonegawa,savin} and their references.

\medskip
In the radial case where $ \Sigma(t) =  \rho(t)S^{n-1}   $,
it is easily checked that  equation \eqref{mcf} reduces to the ODE
$$
\rho'(t) =  -\frac{n-1}{\rho(t)} ,
$$
which yields the  ``ancient'' {\em shrinking sphere solution}
\be\label{ss}
\rho_*(t) =  \sqrt{ -2(n-1) t }, \quad   -\infty < t <   0  .
\ee
The result by Bronsard and Kohn \cite{bk} can be phrased like this:  given a compact interval $I\subset (-\infty, 0)$, there exists
 a radial solution $u^\ve_I (r,t)$  of \eqref{ac1} that satisfies
\eqref{limit} for $t\in I$.

\medskip
In this paper we will construct {\em ancient solutions} to Equation \eqref{ac0},  with one or more transition layers close to the shrinking sphere \eqref{ss}
at all negative times.
Because of self-similarity, we see that the transition layer $|x| = \rho_*(t)$ for a solution $u^\ve$ of  \eqref{ac1} corresponds to the same
region for $u(x,t) = u^\ve (\ve x, \ve^2 t)$, solution of \eqref{ac0}.
Thus in what follows we consider the problem

\be u_t = \Delta u + f(u) \inn \R^n \times (-\infty,0], \quad f(u) = (1-u^2)u.   \label{ac2} \ee
We prove

\begin{theorem} \label{teo1}
There exists a radial solution $u(x,t)$ of equation \eqref{ac2} such that
$$
u(x,t)  =  w( |x| - \rho(t) )   + \phi (x,t)
$$
where
$$
\rho(t) = \sqrt{ -2(n-1) t }  + O(1)  \quad\hbox{as }  t \to - \infty ,
$$
where
$$
\lim_{t\to -\infty} \phi (x,t) = 0 \quad\hbox{uniformly in } x\in \R^n .
$$

\end{theorem}

Our second result extends Theorem \ref{teo1} to the case of ancient solutions with  {\em multiple interfaces}. Given $k\ge 1$, the point is to find solutions of the form
\be
u(x,t)  =  \sum_{i=1}^k   (-1)^{j-1}w ( |x| - \rho_j(t) )  - \frac 12 ( 1+ (-1)^{k-1})  \ +\ \phi(x,t)
\label{forma} \ee
for a lower order perturbation $\phi(x,t)$ as $t\to -\infty$ and  functions
\be
\rho_1(t)  < \rho_2(t) < \cdots <  \rho_k(t) .
\label{nn}\ee
which at main order satisfy  $\rho_j(t) \sim    \sqrt{ -2(n-1) t } $ .  We prove

\begin{theorem} \label{teo2} Given any $k\ge 1$, there exist functions $\rho_j(t)$ as in \eqref{nn}
with
\be {\rho}_j(t)=   \sqrt{-2(n-1)t}  + \frac{1}{\sqrt{2}}\left(j-\frac{k+1}{2}\right)\log\left(\frac{|t|}{\log |t|}\right)+ O(1),\quad j=1,\ldots ,k,\;\; 
\label{forma2}\ee
as $t\to -\infty$,
and a radial ancient solution $u(x,t)$ of equation \eqref{ac2} of the form \eqref{forma} so that
$$
\lim_{t\to -\infty} \phi (x,t) = 0 \quad\hbox{uniformly in } x\in \R^n .
$$

\end{theorem}

The main difference between interfaces and surfaces evolving by mean curvature is that in the phase transition model different components do interact giving rise to interesting motion patterns.
When regarded, after $\ve$-scaling, as a solution of equation \eqref{ac1}, the nodal set of
$u^\ve (x,t) = u(\ve^{-1}x,\ve^{-2}t)$ has $k$ components $\rho_{j\ve} (t)$ which on each compact subinterval of  $(-\infty, 0)$ satisfy

$$ \rho_{j\ve} (t) =  \sqrt{-2(n-1) t }  +  \frac{1}{\sqrt{2}}\left(j-\frac{k+1}{2}\right)\, \ve |\log \ve|   +   o(\ve\log\ve ). $$
The phenomenon described is not present in the limiting flow by mean curvature. Indeed there is a nonlocal interaction between the different components of
the interface that leads to equilibrium. Solutions with multiple interfaces had already been constructed in \cite{dkw4}. In that reference the basic interface is a self translating solution surface of mean curvature flow in $\R^{n-1}$, $n\ge 3$ of the form
$$ x_n = p(|x'|) + t , \quad x'\in \R^{n-1}$$ where $p$ is an entire radially symmetric function (at main order $p(r) \sim r^2).$ Traveling wave solutions
of equation \eqref{ac1} were with multiple-component resembling nested collapsing copies of this ``paraboloid'' were found in \cite{dkw4}.  For a single component, this traveling wave solution was first found in \cite{5authors}. The results of this paper can therefore be regarded as compact analogues of the traveling wave phenomenon. An important difference of is the fact that in our current setting  we cannot reduce the problem to the analysis of an elliptic equation and the parabolic problem must be considered all the way up to time $t=-\infty$.
Interaction of interfaces in the one-dimensional case in this problem has already been considered in \cite{carrpego,chen,fuscohale,dgk}, and in the static
higher dimensional setting in \cite{dkw0,dkw3,dkpw}.
As it will become clear in the course of this paper, the dynamics driving the interaction of the different components of the interface for a solution of the form \eqref{forma}
is given at main order by the first-order Toda type system,

\be
\frac{1}{\xb}\left(\rho_j'+\frac{n-1}{\rho_j}\right)-e^{-\sqrt{2}(\rho_{j+1}-\rho_j)}+e^{-\sqrt{2}(\rho_{j}-\rho_{j-1})}=0,\quad j=1,\ldots ,k,\;\quad\;t\in(-\infty, 0]\label{toda01}
\ee
with the conventions
$\rho_{k+1}=\infty\quad \hbox{and}\quad \rho_0=-\infty,$ and a explicit constant $\beta>0$.
A the proof consists of building by a Lyapunov-Schmidt type procedure a solution.  It is made as a  suitable small perturbation of a first approximation where the functions $\rho_j(t)$ are left as parameters to be determined. The procedure reduces the construction to solving  for the $\rho_j$'s from a system which is a small nonlocal, nonlinear  perturbation of
\eqref{toda1}. We carry out this procedure in the following sections.


\setcounter{equation}{0}
\section{The ansatz}

We will only consider in the proof of Theorem \ref{teo2} the case of an even number $k\ge 2$.
The odd situation (including the case $k=1$ of Theorem \ref{teo1}) is similar.

\medskip
 Setting $r=|x|$ and with some abuse of notation $u= u(t,r)$.
 We want to find a $k$-layer solution to the equation

\be\label{rad} u_t=u_{rr}+\frac{n-1}{r}u_r+f(u), \quad \hbox {for all } \; (t,r)\in (-\infty,-T] \times(0,\infty). \ee
 $$ u_r (t,0) = 0 \qquad \hbox {for all } t\in (-\infty, -T] . $$
 for a large, given $T>0$.
Let
 $w(s)=\tanh(\frac{s}{\sqrt{2}})$ and $k\geq2$ be an even natural number. We set
$$w_j(t,r)=w(r-\rho_j(t)),$$
where the functions $\rho_i(t)$ are ordered,
$$0<\rho_1(t)<\cdots <\rho_k(t).$$
Our purpose is to find a solution of \eqref{ac2} of the form
\be
u(t,r)=\sum_{j=1}^{k}(-1)^{j-1}w_j(t,r)-1+\psi(t,r).\label{form}
\ee
where the functions $\rho_j(t)$ are required to satisfy at main order the system 

\be
\frac{1}{\xb}\left(\rho_j'+\frac{n-1}{\rho_j}\right)-e^{-\sqrt{2}(\rho_{j+1}-\rho_j)}+e^{-\sqrt{2}(\rho_{j}-\rho_{j-1})}=0,\quad j=1,\ldots ,k,\;\quad\;t\in(-\infty, -T]\label{toda1}
\ee
with the conventions
$\rho_{k+1}=\infty\quad \hbox{and}\quad \rho_0=-\infty,$ and a explicit constant $\beta>0$ given by \eqref{beta} below.  
We will prove in Section 5 that system \eqref{toda1} has a solution with the following form 
\be\label{forma4}
\rho_j(t) = \rho^0_j(t)  + h_j(t)
\ee
where $h_j(t) = O( (\log|t|)^{-1}) $ as $t\to-\infty$ and   $\rho^0_j(t) $ takes the form 
\be \label {fff}
\rho^0_j(t) = \sqrt{-2(n-1)t} + (j-\frac{k+1}{2})\eta +\xg_j
\ee
where the $\xg_j$ are explicit constants (given in Lemma \ref{lemodesimpl2}) and $\eta(t)$ solves the ODE  
\begin{align}
\eta'+\frac{1}{2t}\eta+ e^{-\sqrt{2}\eta}&=0,\quad t\in(-\infty,-1]\\
\eta(-1)&=0,\label{odesimp21}
\end{align}
which according to Lemma \ref{lemodesimpl1}
satisfies as $t\to -\infty$,

\begin{align}
\eta(t) =   \frac{1}{\sqrt{2}}\log\left(\frac{|t|}{\log |t|}\right) + O(1).
\end{align}
and $\xg_j$ are the constants defined in Lemma \ref{lemodesimpl2}. 
Let us set $\rho(t)=(\rho_1(t),\ldots ,\rho_k(t))$ and  write
\be
\rho(t):=\rho^0(t)+h(t),\label{ksij}
\ee
where the $\rho^0_j$'s are the functions in  \eqref{fff} and the (small) functions $h_j(t)$ are parameters to be found, on which we only a priori assume
$$ \sup_{t\leq -2}|{h}(t)|+\sup_{t\leq -2}\frac{|t|}{\log|t|}|{h}'(t)|<1.  $$ 
We look for a solution of equation \eqref{rad} of the form \eqref{form}.
We set 
\be
z(t,x)=\sum_{j=1}^{k}(-1)^{j+1}w_j(t,x)-1\label{z}
\ee
and consider the following projected version of equation \eqref{rad} in terms of $\psi$:  

\begin{align}\nonumber
\psi_t&=\psi_{rr}+\frac{n-1}{r}\psi_r+f'(z(t,r))\psi+E+N(\psi)\\ &-\sum_{i=1}^k c_i(t)w'(r-\rho_i(t)),\qquad\qquad\mathrm{in}\;\;(-\infty,-T)\times(0,\infty)\\
\label{mainpro}
\end{align}
and
\be
\int_0^\infty r^{n-1}\psi(t,r)w'(r-\rho_i(t))dr=0,\quad\hbox{for all }  i=1,\ldots ,k,\;t<-T. \label{orthcond}
\ee
where
\begin{align}\nonumber
E=\sum_{j=1}^{k}(-1)^{j+1}&\left(w'(r-\rho_j(t))\rho_j'(t)+\frac{n-1}{r}w'(r-\rho_j(t))\right)+f(z(t,r))\\
&-\sum_{j=1}^{k}(-1)^{j+1}f(w_j(t,r)), \label{error}
\end{align}
$$N(\psi)=f(\psi(t,r)+z(t,r))-f(z(t,r))-f'(z(t,r))\psi,$$
where the functions $c_i(t)$ are chosen so that $\psi$ satisfies the orthogonality condition (\ref{orthcond}), namely in such a way that the following (nearly diagonal) system holds.
\begin{align}\nonumber
\sum_{i=1}^kc_i(t)\int_0^\infty &w'(r-\rho_i(t))w'(r-\rho_j(t))r^{n-1}dr\\ \nonumber
&=-\int_0^\infty \psi_{r}w''(r-\rho_j(t))r^{n-1}dr+\int_0^\infty f'(z(t,r))\psi w'(r-\rho_j(t)) r^{n-1}dr\\ \nonumber
&-\rho'_j(t)\int_0^\infty\psi(t,r)w''(r-\rho_j(t))r^{n-1}dr\\
&+\int_0^\infty(E+N(\psi))w'(r-\rho_j(t))r^{n-1}dr,\qquad\forall i=1,...,k,\;t<-T.\label{ole}
\end{align}

Later we will choose $h(t)$ such that $c_i(t)=0,\;\forall\;i=1,...,k.$
In the following lemma we find a bound for the error term $E=E(t,r)$ in \eqref{error}.
\begin{lemma}
Let $T_0>1,$ $0<\xs<\sqrt{2},$ we define
\bea\nonumber
\xF(t,r)&=&e^{\xs(-r+\rho_{j-1}^0(t))}+e^{\xs(r-\rho_{j+1}^0(t))},\\ \nonumber
&&\mathrm{if}\;\;\qquad\frac{\rho_j^0(t)+\rho_{j-1}^0(t)}{2}\leq r\leq \frac{\rho_j^0(t)+\rho_{j+1}^0(t)}{2},\;j=2,...,k,\\ \nonumber
\xF(t,r)&=&e^{\xs(r-\rho_{2}^0(t))},\quad \text{if}\quad \frac{\rho^0_0(t)+\rho_{1}^0(t)}{2}\leq r\leq \frac{\rho_1^0(t)+\rho_{2}^0(t)}{2}\\
\xF(t,r)&=&e^{\xs(r-\rho_{1}^0(t))},\quad \text{if}\quad r\leq \frac{\rho^0_0(t)+\rho_{1}^0(t)}{2}
\label{bound}
\eea
with $\rho_0^0=\rho_1^0-\eta$ and $\rho_{k+1}^0=\infty.$ Then there exists a uniform constant $C>0$ which depends only on $k,$ such that
$$|E(t,r)|\leq C(1+\frac{1}{r})\xF(t,r),\quad\forall (t,r)\in(-\infty, -T_0]\times(0,\infty),$$
where $E$ is the error term in \eqref{error}, and $\rho$ satisfies the assumptions of this section.\label{remark}
\end{lemma}
\begin{proof}

First we note that
$$|\rho_j'(t)+\frac{n-1}{r}|\leq C\frac{\log |t|}{|t|},\quad\text{if}\;\; \frac{\rho_1^0(t)+\rho_{0}^0(t)}{2}\leq r\leq \rho_k^0+\frac{\sqrt{2}+\xs}{\sqrt{2}-\xs}\eta,$$

$$\frac{w'(r-\rho_j^0(t))}{\xF}\leq C\left(\frac{\log |t|}{|t|}\right),\quad\forall\;\; r\geq \rho_k^0+\frac{\sqrt{2}+\xs}{\sqrt{2}-\xs}\eta $$
and
$$\frac{w'(r-\rho_j(t))}{\xF}\leq C\left(\frac{\log{|t|}}{|t|}\right)^{-\frac{\xs}{\sqrt{2}}},\quad\forall r>0,$$
for some positive constant independent of $t,$ $T_0$ and $r.$

Next assume that
$$\frac{\rho_j^0(t)+\rho_{j-1}^0(t)}{2}\leq r\leq \frac{\rho_j^0(t)+\rho_{j+1}^0(t)}{2},\;j=1,...,k.$$

If $i\leq j-1,$ by our assumptions on $\rho_i,$ there exists a uniform constant $C>0$ such that

$$|w(r-\rho_i(t))-1|\leq Ce^{\sqrt{2}(-r+\rho_{j-1}^0(t))}.$$
Similarly if $i\geq j+1$

$$|w(r-\rho_i(t))+1|\leq Ce^{\sqrt{2}(r-\rho_{j+1}^0(t))}.$$

We set
$$g=\sum_{i=1}^{j-1}(-1)^{i+1}\left(w(r-\rho_i)-1\right)+\sum_{i=j+1}^k(-1)^{i+1}\left(w(r-\rho_i)+1\right).$$

Then
\begin{align}\nonumber
&\left|f\left(g+(-1)^{j+1}w(r-\rho_j(t))\right)-\sum_{i=1}^{k}(-1)^{i+1}f(w_i(t,r))\right|\\ \nonumber
&\leq C\left(\sum_{i=1}^{j-1}|w(r-\rho_i)-1|+\sum_{i=j+1}^k(-1)^{i+1}|w(r-\rho_i)+1|\right).
\end{align}
Combining all above and using the properties of $\rho$ we can reach to the desired result.
\end{proof}

\setcounter{equation}{0}
\section{The linear problem}

This section is devoted to build a solution to the linear parabolic problem

\be
\psi_t=\;\psi_{rr}+\frac{n-1}{r}\psi_r+f'(z(t,r))\psi+g(t,x)-\sum_{j=1}^k c_i(t)w'(r-\rho_j(t)),\qquad \mathrm{in}\;\;\;(-\infty,-T_0]\times(0,\infty).\label{fixprop}
\ee

\begin{align}
\int_{\mathbb{R}}r^{n-1}\psi(t,r)w'(r-\rho_i(t))dr&=0,\qquad\forall i=1,...,k,\ t\in (-\infty,-T_0]  \label{orthcond1}
\end{align}
for a bounded function $g,$ and  $T_0>0$ fixed sufficiently large. In this section we use the following notations
\begin{notation}
i)
$$\rho=\rho^0+h,$$
ii)
$$z(t,x)=\sum_{j=1}^{k}(-1)^{j+1}w(x-\rho_j(t))-1,$$
where $h:\mathbb{R}\mapsto\mathbb{R}^k$ is a function that satisfies $$\sup_{t\leq -2}|h(t)|+\sup_{t\leq -2}\frac{|t|}{\log|t|}|h'(t)|<1.$$\label{notation}
\end{notation}

The numbers $c_i(t)$ are exactly those that make the relations above consistent,
namely, by definition for each $t<-T_0$ they solve the linear system of equations
\begin{align}\nonumber
\sum_{i=1}^kc_i(t)\int_0^\infty&w'(r-\rho_i(t))w'(r-\rho_j(t))r^{n-1}dr\\ \nonumber
&=-\int_0^\infty\psi_{r}w''(r-\rho_j(t)) r^{n-1}dr+\int_0^\infty f'(z(t,r))\psi w'(r-\rho_j(t)) r^{n-1}dr\\ \nonumber
&-\rho'_j(t)\int_0^\infty\psi(t,r)w''(r-\rho_j(t))r^{n-1}dr\\
&+\int_0^\infty g(t,r)w'(r-\rho_j(t))r^{n-1}dr,\qquad\forall i=1,...,k,\;t<-T.
\end{align}
This system can indeed be solved uniquely since if $T_0$ is taken sufficiently large, the matrix  with coefficients $\int_{\mathbb{R}}r^{n-1}w'(r-\rho_i(t))w'(r-\rho_j(t))dr$
is nearly diagonal.

Our purpose is to build a linear operator $\psi = A(g)$ that defines a solution of \eqref{fixprop}-\eqref{orthcond1} which is bounded for norm suitably adapted to our setting.

\medskip

Let $\mathcal{C}_\xF((s,t)\times(0,\infty)$ is the space of continuous functions
with norm
\be
||u||_{\mathcal{C}_\xF((s,t)\times(0,\infty))}=\left|\left|\frac{u}{\xF}\right|\right|_{L^\infty((s,t)\times(0,\infty))},\label{norm11111}
\ee
where $\xF$ has been defined in \eqref{bound}.
\begin{prop}  \label{prop1} Let $g=g_1/r+g_2.$ There exist positive numbers $T_0$ and $C$ such that
for each $g_1, g_2\in\mathcal{C}_\xF((-\infty ,0)\times\mathbb{R}),$ there exists a solution
of Problem \eqref{fixprop}-\eqref{orthcond1}  $\psi = A(g)$ which defines
a linear operator of $g$ and satisfies the estimate
\be
||\psi||_{\mathcal{C}_\xF((-\infty,t)\times(0,\infty))}\leq C\left(||g_1||_{\mathcal{C}_\xF((-\infty,t)\times(0,\infty))}+||g_2||_{\mathcal{C}_\xF((-\infty,t)\times(0,\infty))}\right),\qquad\forall t\leq -T_0.
\label{estfix**}
\ee\label{fixth}
\end{prop}

The proof will be a consequence of intermediate steps that we state and prove next.
Let $g(t,r)\in \mathcal{C}_\xF\left((-\infty,-T)\times(0,\infty)\right).$ For $T>0$ and $s<-T$ we consider the Cauchy problem
\bea\nonumber
\psi_t&=&\;\psi_{rr}+\frac{n-1}{r}\psi_r+f'(z(t,r))\psi+g(t,r),\qquad \mathrm{in}\;\;\;(s,-T]\times(0,\infty),\\ \nonumber
\qquad\qquad\qquad\psi(s,r)&=&0,\phantom{\psi_{xx}+f'(z)\psi+g(t,x)}\qquad \mathrm{in}\;\;\;(0,\infty)\\
\lim_{r\rightarrow0} r^{n-1}\psi_r(t,r)&=&0,\quad\forall t\in(s,-T]\label{mainpro*}
\eea
which is uniquely solvable. We call $\psi^s(t,r)$ its solution.
\subsection{A priori estimates for the solution of the problem (\ref{mainpro*})}

 We will establish in this subsection a priori estimates for the
solutions  $\psi^s$ of (\ref{mainpro*}) that are independent on $s.$
\begin{lemma}

Let $g=g_1/r+g_2,$ $g_1,g_2\in \mathcal{C}_\xF\left((s,-T)\times(0,\infty)\right)$ and $\psi^s\in \mathcal{C}_\xF\left((s,-T)\times(0,\infty)\right)$ be a solution of the problem (\ref{mainpro*}) which satisfies the orthogonality conditions

\be
\int_0^\infty r^{n-1}\psi^s(t,r)w'(r-\rho_i(t))dr=0,\qquad\forall i=1,...,k,\;s<t<-T. \label{cond}
\ee

Then there exists a uniform constant $T_0>0$ such that for any $t\in (s,-T_0],$ the following estimate is valid
\be
||\psi^s||_{\mathcal{C}_\xF\left((s,t)\times(0,\infty)\right)}\leq C\left(||g_1||_{\mathcal{C}_\xF\left((s,t)\times(0,\infty)\right)}+||g_2||_{\mathcal{C}_\xF\left((s,t)\times(0,\infty)\right)}\right).\label{estfix}
\ee
where $C>0$ is a uniform constant.\label{mainlem}
\end{lemma}
\begin{proof}
Set
$$A_j^{(s,t)}=\left\{(\tau,r)\in(s,t)\times(0,\infty):\;\frac{\rho_j^0(\tau)+\rho_{j-1}^0(\tau)}{2}< r< \frac{\rho_j^0(\tau)+\rho_{j+1}^0(\tau)}{2}\right\},$$
with $\rho_0^0=\rho_1^0-\eta$ and $\rho_{k+1}^0=\infty,$ and
$$A_{j,R}^{(s,t)}=\left\{(\tau,x)\in(s,t)\times(0,\infty):\;|r-\rho_j^0(\tau)|<R+1\right\}.$$

We will prove \eqref{estfix} by contradiction. Let $\{s_i\},\;\{\overline{t}_i\}$ be sequences such that $s_i< \overline{t}_i\leq -T_0,$ and
$s_i\downarrow-
\infty,$ $\overline{t}_i\downarrow-\infty.$ We assume that there exists $g_{i}=g_{1,i}/r+g_2$ such that $\psi_i$  solve (\ref{mainpro*}) with $s=s_i,$ $-T=\overline{t}_i,$ $g=g_i$ and satisfies \eqref{cond}.

Finally we assume that
\be
\left|\left|\psi_i\right|\right|_{\mathcal{C}_\xF((s_i,\overline{t}_i)\times(0,\infty))}=1,\label{contr}
\ee
$$\left|\left|g_{1,i}\right|\right|_{\mathcal{C}_\xF((s_i,\overline{t}_i)\times(0,\infty))}
+\left|\left|g_{2,i}\right|\right|_{\mathcal{C}_\xF((s_i,\overline{t}_i)\times(0,\infty))}\rightarrow0,$$

First we note that we can assume
$$s_i+1<\overline{t}_i.$$
Indeed, set
$$\xf_j(t,r)=M C(g_i)e^{l (t-s_i)}\left(e^{\xs\left(-r+\rho_{j+1}^0(t)\right)}+e^{\xs\left(r-\rho_{j-1}^0(t)\right)}+e^{\xs\left(4-r-\rho_{1}^0(t)\right)}\right),$$
where
$$C(g_i)=\left|\left|g_{1,i}\right|\right|_{\mathcal{C}_\xF((s_i,\overline{t}_i)\times(0,\infty))}
+\left|\left|g_{2,i}\right|\right|_{\mathcal{C}_\xF((s_i,\overline{t}_i)\times(0,\infty))}.$$

If we choose $M,l>0$ large enough, we can use $\xf_j$ like barrier to obtain
\be |\psi_i(t,r)|\leq C e^{l (t-s_i)}\xF(t,r)\left(\left|\left|g_{1,i}\right|\right|_{\mathcal{C}_\xF((s_i,\overline{t}_i)\times(0,\infty))}
+\left|\left|g_{2,i}\right|\right|_{\mathcal{C}_\xF((s_i,\overline{t}_i)\times(0,\infty))}\right).\label{4444}\ee
Thus by above inequality we can choose $s_i+1<\overline{t}_i.$

To reach at contradiction we need the following assertion,

\textbf{Assertion 1.} \emph{Let $R>0$ then we have}
\be
\lim_{i\rightarrow\infty}\left|\left|\frac{\psi_i}{\xF}\right|\right|_{L^\infty(A_{j,R}^{(s_i,\overline{t}_i)})}=0,\;\;\;\;\;\forall j=1,...k.\label{contr2}
\ee
Let us first assume that \eqref{contr2} is valid.

Set
$$
\xm_{i,j}:=\left|\left|\frac{\psi_i}{\xF}\right|\right|_{L^\infty(A_{j,R}^{(s_i,\overline{t}_i)})}\longrightarrow_{i\rightarrow\infty}0,\;\;\;\;\;\forall j=1,...k.
$$

Let $$\frac{\rho_j^0(t)+\rho_{j-1}^0(t)}{2}\leq x\leq \frac{\rho_j^0(t)+\rho_{j+1}^0(t)}{2},\;j=1,...,k$$ with $\rho_0^0=\rho_1^0-\eta$ and $\rho_{k+1}^0=\infty.$

If $n\leq j-1,$ then we have by our assumptions on $\rho_n$

$$|w(r-\rho_n(t))-1|\leq Ce^{\sqrt{2}(-r+\rho_{n-1}(t))}\leq Ce^{-\frac{\sqrt{2}}{2}(\rho_{j}-\rho_{j-1}(t))} \leq C\left(\frac{\log|t|}{|t|}\right)^{\frac{1}{2}}.$$
Similarly if $n\geq j+1$

$$|w(r-\rho_n(t))+1|\leq2e^{\sqrt{2}(r-\rho_{n+1}(t))}\leq C\left(\frac{\log|t|}{|t|}\right)^{\frac{1}{2}}.$$

Moreover if we assume that $|r-\rho_j(t)|>R+1,$ then we have that
$$|w(r-\rho_j(t))|\geq w(R).$$
Combining all above
for any $0<\xe<\sqrt{2}$ there exists $i_0\in \mathbb{N}$ and $R>0$ such that

\be
-f'(z(r,x))\geq 2-\xe^2,\;\;\forall t\leq \overline{t}_i,\;\;x\in \mathbb{R}\setminus\cup_{j=1}^kA_{j,R}^{(s_i,t_i)}\;\;\; \text{and}\;\;\;i\geq i_0.\label{fr11}
\ee
Consider the function
\begin{align*}
\overline{\xf}_{i,j}(t,r)&=M\left(e^{\xs\left(-r+\rho_{j+1}^0(t)\right)}+e^{\xs\left(r-\rho_{j-1}^0(t)\right)}+e^{\xs\left(4\tilde{M}-r-\rho_{1}^0(t)\right)}\right)\\
&\times\left(\left|\left|g_{1,i}\right|\right|_{\mathcal{C}_\xF((s_i,\overline{t}_i)\times(0,\infty))}
+\left|\left|g_{2,i}\right|\right|_{\mathcal{C}_\xF((s_i,\overline{t}_i)\times(0,\infty))}
+\sup_{1\leq j\leq k}\xm_{i,j}\right),
\end{align*}
where $M,\tilde{M}>1$  is large enough which does not depend on $s_i, \overline{t}_i.$

First we note that
$$\max(\psi_i(t,x)-\xf_{i,j}(t,x),0)=0,\quad\forall\; (t,x)\in \overline{\cup_{j=1}^kA_{j,R}^{(s_i,\overline{t}_i)}}.$$

Now, let $\xe>0,\;\tilde{M}>1$ be such that $\frac{n-1}{\tilde{M}}+2-\xe^2>\xs^2.$ Then we can choose $i_0$ such that for any $i> i_0,$ we can use $\xf_{i,j}$ like a barrier to obtain

$$|\psi_i(t,r)|\leq |\overline{\xf}_{i,j}(t,r)|,\;\;\forall(t,r)\in ((s_i,\overline{t}_i)\times(0,\infty))\setminus\cup_{j=1}^kA_{j,R}^{(s_i,\overline{t}_i)},\;\;j=1,...k,\;\;\;i\geq i_0.$$
The above inequality implies
$$1=\left|\left|\psi_i\right|\right|_{\mathcal{C}_\xF((s_i,\overline{t}_i)\times(0,\infty))}\leq
M\left(\left|\left|g_i\right|\right|_{\mathcal{C}_\xF((s_i,\overline{t}_i)\times(0,\infty))}
+\sup_{1\leq j\leq k}\xm_{i,j}\right),$$
which is clearly a contradiction if we choose $i$ large enough.

\noindent\textbf{Proof of Assertion 1.}
We will prove Assertion 1 by contradiction in four steps.

Let us give first the contradict argument and some notations.  We assume that (\ref{contr2}) is not valid. Then there exists $j\in\{1,...,k\}$ and $\xd>0$ such that
$$\left|\left|\frac{\psi_i}{\xF}\right|\right|_{L^\infty(A_{j,R}^{(s_i,\overline{t}_i)})}>\xd>0,\;\;\forall i\in\mathbb{N}.$$
Let $(t_i,y_i)\in A_{j,R}^{(s_i,\overline{t}_i)}$ such that
\be
\left|\frac{\psi_i(t_i,y_i)}{\xF(t_i,y_i)}\right|>\xd.\label{667}
\ee

We observe here that by definition of $\xF$
\be
\xF(t_i,y_i)=e^{\xs(-y_i+\rho_{j-1}(t_i))}+e^{\xs(y_i-\rho_{j+1}(t_i))}.\label{obs1}
\ee

We set $r=x+\rho_j(t+t_i),\;\;y_i=x_i+\rho_j(t_i)$ and
$$\xf_i(t,x)=\frac{\psi_i(t+t_i,x+x_i+\rho_j(t+t_i))}{\xF(t_i,x_i+\rho_j(t_i))}.$$

Then $\xf_i$ satisfies
\bea\nonumber
(\xf_i)_t&=&\; (\xf_i)_{xx}+\frac{n-1}{x+x_i+\rho_j(t+t_i)}(\xf_i)_x\\ \nonumber
&+&\rho_j'(t+t_i)(\xf_i)_x+f'(z(t+t_i,x+x_i+\rho_j(t+t_i)))\xf_i\\ \nonumber&+&\frac{g_i(t+t_i,x+x_i+\rho_j(t+t_i))}{\xF(t_i,x_i+\rho_j(t_i))}, \quad \mathrm{in}\;\;\;\xG^{(s_i,t_i)}_j,\\
\xf_i(s_i-t_i,x)&=&0,\phantom{\psi_{xx}+f'(z)\psi+g(t,x)}\qquad \mathrm{in}\;\;\;(-x_i-\rho_j(t+t_i),\infty),\label{eq11}
\eea
where $$\xG^{(s_i,t_i)}_j=\{(t,x)\in(s_i,t_i]\times\mathbb{R}:\;-x_i-\rho_j(t+t_i)<x\}.$$

Also set
\begin{align}
\nonumber
&B_{t_i,n,j}=\Bigg\{(t,x)\in(s_i-t_i,0]\times\mathbb{R}:
\;\frac{\rho_n^0(t+t_i)+\rho_{n-1}^0(t+t_i)}{2}-\rho_j(t+t_i)-x_i\\ \nonumber
&\leq x\leq
 \frac{\rho_n^0(t+t_i)+\rho_{n+1}^0(t+t_i)}{2}-\rho_j(t+t_i)-x_i\Bigg\}
\end{align}
and
\begin{align*}
B_{t_i,n,j}^M=B_{t_i,n,j}\cap\left\{(t,x)\in(s_i-t_i,0]\times\mathbb{R}: |x+\rho_j(t+t_i)+x_i-\rho_n^0(t+t_i)|>M\right\}
\end{align*}
where $n=1,....,k$ and $M>0.$
We note here that $|x_i|<R+1, \;\forall\; i\in\mathbb{N},$ $|\xf_i(0,0)|=\left|\psi_i(t_i,y_i)/\xF(t_i,y_i)\right|>\xd>0.$
Also in view of the proof of \eqref{4444} and the assumption \eqref{667} we can assume that
$$\liminf t_i-s_i>\infty.$$
Without loss of generality we assume that $x_i\rightarrow x_0\in B_{R+1}(0),$ $\lim_{i\rightarrow\infty} t_i-s_i=\infty$ (otherwise take a subsequence).

\medskip
\noindent\textbf{Step 1}

We assert that
$\xf_i\rightarrow\xf$ locally uniformly, $\xf(0,0)>\xd$  and $\xf$ satisfies
\be
\xf_t=\; \xf_{xx} +f'(w(x+x_0))\xf,\qquad \mathrm{in}\;\;\;(-\infty,0]\times\mathbb{R}.\label{equ1}
\ee

Let $(t,x)\in B_{t_i,n,j},\;1\leq n\leq k.$ By \eqref{ksij}, \eqref{contr} and \eqref{obs1} we have that
\begin{align}\nonumber
&|\xf_i(t,x)|\leq \left|\frac{\psi_i(t+t_i,x+x_i+\rho_j(t+t_i))}{\xF(t_i,x_i+\rho_j(t_i))}\right|\\ \nonumber
&\leq \left|\frac{\xF(t+t_i,x+x_i+\rho_j(t+t_i))}{\xF(t_i,x_i+\rho_j(t_i))}\right|\\ \nonumber
&\leq C_0(\xb,||h||_{L^\infty},\sup_{1\leq j\leq k}|\xg_j|,\xs,R)\\ &\times\left(\frac{|t_i|\log|t+t_i|}{|t+t_i|\log|t_i|}\right)^{\frac{\xs}{\sqrt{2}}}e^{\xs|x+\rho_j(t+t_i)-\rho_n(t+t_i)|},\quad\forall i\in\mathbb{N},\;(t,x)\in B_{t_i,n,j}.
\label{constant}
\end{align}

Now note here that $$\cup_{i=1}^\infty B_{t_i,j,j}=(-\infty,0]\times\mathbb{R}.$$

Thus the proof of the assertion of this step is complete.

\medskip
\noindent\textbf{Step 2} In this step we prove the following orthogonality condition for $\xf$.

\be
\int_{\mathbb{R}}\xf(t,x)w'(x+x_0)dx=0,\;\;\;\forall t\in (-\infty,0].\label{oth}
\ee

Let $t\in \cap_{i=i_0}^\infty(s_i-t_i,0],$ for some $i_0\in\mathbb{N}.$

\begin{align}
\nonumber
x\in &B_{t,t_i,n,j}=\Bigg\{x\in\mathbb{R}:
\;\frac{\rho_n^0(t+t_i)+\rho_{n-1}^0(t+t_i)}{2}-\rho_j(t+t_i)-x_i\\ \nonumber
&\leq x\leq
 \frac{\rho_n^0(t+t_i)+\rho_{n+1}^0(t+t_i)}{2}-\rho_j(t+t_i)-x_i\Bigg\}
\end{align}
By (\ref{constant}) we have that
\be
\left|\int_{B_{t,t_i,j,j}}x^\xa\xf_i(t,x)w'(x+x_i)dx\right|\leq C_0\int_{\mathbb{R}}r^{\xa}e^{-(\sqrt{2}-\xs)|x|}dx< C,\quad\forall \xa=0,1,...,n-1.\label{est2}
\ee

Let $\xa\in\mathbb{N}\cup\{0\}$ and $n>j.$ By \eqref{constant}, the assumptions on $\rho$ (see Notation \ref{notation}) and the fact that $|x_i|<R+1$ we have that

\begin{align}\nonumber
&\left|\int_{B_{t,t_i,n,j}}x^\xa\xf_i(t,x)w'(x+x_i)dx\right|\\ \nonumber
&\leq C_0 \int_{{\frac{\rho_n^0(t+t_i)+\rho_{n-1}^0(t+t_i)}{2}-\rho_j(t+t_i)-x_i}}^{\frac{\rho_n^0(t+t_i)+
\rho_{n+1}^0(t+t_i)}{2}-\rho_j(t+t_i)-x_i}|x|^\xa e^{-\sqrt{2}x+\xs|x+\rho_j(t+t_i)-\rho_n(t+t_i)|}dx\\
&\leq C (\log|t+t_i|)^{\xa}\left(\frac{\log|t+t_i|}{|t+t_i|}\right)^{\frac{(\sqrt{2}-\xs)}{2\sqrt{2}}}\rightarrow_{i\rightarrow\infty}0.\label{est1}
\end{align}
Similarly the estimate \eqref{est1} is valid if $n<j.$

By (\ref{est2}), (\ref{est1}) we have that
\begin{align*}
0&=\lim_{i\rightarrow\infty}\frac{1}{(\rho_j(t+t_i))^{n-1}}\int_{-x_i-\rho_j(t+t_i)}^\infty(x+\rho_j(t+t_i)+x_i)^{n-1}\xf_i(t,x)w'(x+x_i)dx\\
&=\int_{\mathbb{R}}\xf(t,x)w'(x+x_0)dx
\end{align*}
and the proof of this assertion follows.

\medskip
\noindent\textbf{Step 3} In this step we prove the following assertion:

 \emph{There exists $C=C(R,\xs)>0,$ such that}

\be
|\xf(t,x)|\leq Ce^{-\xs|x|},\;\;\forall(t,x)\in(-\infty,0]\times \mathbb{R}.\label{bound1}
\ee

Now, note that if $(t,x)\in B_{t_i,n,j},$ by definition of $\rho$ (Notation \ref{notation}), we have
$$
e^{\xs|x+\rho_j(t+t_i)-\rho_n(t+t_i)|}\leq C_0(\xb,||h||_{L^\infty},\sup_{1\leq j\leq k}|\xg_j|,\xs,R) e^{\xs|x|}.
$$
Thus, in view of the proof of \eqref{constant} we have that
$$
\left|\frac{g_i(t+t_i,x+x_i+\rho_j(t+t_i))}{\xF(t_i,x_i+\rho_j(t_i))}\right|\leq C C_1(g_i)
e^{\xs|x|},\;\; \forall x\geq-x_i-\rho_j(t+t_i)+\tilde{M},\;\forall i\in\mathbb{N}
$$
and
\begin{align*}
\left|\frac{g_i(t+t_i,x+x_i+\rho_j(t+t_i))}{\xF(t_i,x_i+\rho_j(t_i))}\right|&\leq C C_1(g_i)\frac{\left|\left|g_i\right|\right|_{\mathcal{C}_\xF((s_i,\overline{t}_i)\times\mathbb{R})}}{x+x_i+\rho_j(t+t_i)}
e^{\xs|x|}\\
&\;\;, \forall -x_i-\rho_j(t+t_i)<x\leq-x_i-\rho_j(t+t_i)+\tilde{M},\;\forall i\in\mathbb{N}
\end{align*}
where
$$C_1(g_i)=\left|\left|g_{1,i}\right|\right|_{\mathcal{C}_\xF((s_i,\overline{t}_i)\times(0,\infty))}
+\left|\left|g_{2,i}\right|\right|_{\mathcal{C}_\xF((s_i,\overline{t}_i)\times(0,\infty))}.$$
Let $\xe>0$ be such that $\xs+\xe<\sqrt{2},$ set $$G(t,x)=C(M)\left(e^{-\xs|x|}+\left|\left|g_i\right|\right|_{\mathcal{C}_\xF((s_i,\overline{t}_i)\times(0,\infty))}
\left(e^{(\xs+\xe)x}+e^{-(\xs+\xe)x}\right)\right).$$

In view of the proof of Assertion 1 we can find $i_0,$ $R,$ $\tilde{M}$ and $M>0$ such that we use $G(t,x)$ like a barrier to obtain

$$\xf_i\leq G(t,x),\quad\forall (t,x)\in \xG^{(s_i,t_i)}_j.$$

The proof of \eqref{bound1} follows if we send $i\rightarrow\infty.$

\medskip
\noindent\textbf{ Step 4}
In this step we prove the assertion \eqref{contr2}.
Consider the Hilbert space $$H=\{\xz\in H^1(\mathbb{R}):\;\int_{\mathbb{R}}\xz(x)w'(x)dx=0.\}$$
Then it is well known that the following inequality is valid
\be
\int_{\mathbb{R}}|\xz'(x)|^2-f'(w)|\xz|^2\geq c\int_{\mathbb{R}}|\xz(x)|^2dx, \qquad\forall\xf(x)\in H\cap L^2(\mathbb{R}).\label{Heq}
\ee

Thus if we multiply (\ref{equ1}) by $\xf$ and integrate with respect $x$ we have
\bea\nonumber
0&=&\frac{1}{2}\int_{\mathbb{R}}(\xf^2)_tdx+\int_{\mathbb{R}}|\xf_x|^2-f'(w(x))|\xf|^2\\ \nonumber
&\geq&\frac{1}{2}\int_{\mathbb{R}}(\xf^2)_tdx+c\int_{\mathbb{R}}|\xf(t,x)|^2dx.
\eea

Set $a(t)=\int_{\mathbb{R}}|\xf(t,x)|^2dx,$ we have that there exists a $c_0$ such that
$$a'(t)\leq-c_0a(t)\Rightarrow a(t)>a(0)e^{c_0|t|},$$
which is a contradiction since
$$
||e^{\xs|x|}\xf||_{L^\infty((-\infty,0]\times\mathbb{R})}<C.
$$

\end{proof}

\subsection{The problem \eqref{mainpro*} with $g(t,r)=h(t,r)-\sum_{j=1}^k c_i(t)w'(r-\rho_j(t))$}
In this subsection, we study the following problem.

\bea\nonumber
\psi_t&=&\;\psi_{rr}+\frac{n-1}{r}\psi_r\\ \nonumber
&+&f'(z(t,r))\psi+h(t,r)-\sum_{j=1}^k c_i(t)w'(r-\rho_j(t)),\qquad \mathrm{in}\;\;\;(s,-T]\times(0,\infty),\\ \nonumber
\qquad\qquad\qquad\psi(s,r)&=&0,\phantom{\psi_{xx}+f'(z)\psi+g(t,x)}\quad \mathrm{in}\;\;\;(0,\infty)\\
\lim_{r\rightarrow0} r^{n-1}\psi_r(t,r)&=&0,\quad\forall t\in(s,-T]\label{proci}
\eea

where $h=h_1/r+h_2,$ $h_1,h_2\in\mathcal{C}_\xF((s,-T)\times\mathbb{R})$ and $c_i(t)$ satisfies the following (nearly diagonal) system
\begin{align}\nonumber
\sum_{i=1}^kc_i(t)\int_0^\infty&w'(r-\rho_i(t))w'(r-\rho_j(t))r^{n-1}dr\\ \nonumber
&=-\int_0^\infty r^{n-1}\psi_{r}w''(r-\rho_j(t))dr+\int_0^\infty f'(z(t,r))\psi w'(r-\rho_j(t)) r^{n-1}dr\\ \nonumber
&-\rho'_j(t)\int_0^\infty\psi(t,r)\left(w'(r-\rho_j(t))r^{n-1}\right)_rdr\\
&+\int_0^\infty h(t,r)w'(r-\rho_j(t))r^{n-1}dr,\qquad\forall i=1,...,k,\;t<-T_0.\label{ci(t)}
\end{align}
We note here that if $\psi$ is a solution of \eqref{proci} and $c_i(t)$ satisfies the above system then $\psi$ satisfies the orthogonality conditions $$\int_{0}^\infty\psi(t,r)w'(r-\rho_i(t))r^{n-1}dr=0,\quad\forall i=1,...,k,\;s<t<-T_0.$$

\medskip
The main result of this subsection is the following
\begin{lemma}
Let $h=h_1/r+h_2,$ $h_1,h_2\in\mathcal{C}_\xF((s,-T)\times\mathbb{R})$. Then there exist a uniform constant $T_0\geq T>0,$ and a unique solution  $\psi^s$ of the problem \eqref{proci}.

Furthermore, we have that $\psi^s$ satisfies the orthogonality conditions (\ref{orthcond}), $\forall s<t<-T_0,$ and the following estimate
 \be
||\psi^s||_{\mathcal{C}_\xF((s,t)\times(0,\infty))}\leq C\left(||h_1||_{\mathcal{C}_\xF((s,t)\times(0,\infty))}
+||h_2||_{\mathcal{C}_\xF((s,t)\times(0,\infty))}\right).\label{estfixmef}
\ee
where $C>0$ is a uniform constant.\label{cilemma*}
\end{lemma}
To prove the above Lemma we need the following result

\begin{lemma}
Let $T>0$ big enough, $h=h_1/r+h_2,$ $h_1,h_2\in\mathcal{C}_\xF((s,-T)\times\mathbb{R})$ and $\psi\in \mathcal{C}_\xF((s,-T)\times(0,\infty)).$ Then there exist $c_i(t),\;i=1,...,k$ such that the nearly diagonal system \eqref{ci(t)} holds.

Furthermore the following estimates for $c_i$ are valid, for some constant $C>0$ that does not depends on $T,\;s,\;t,\;\psi,\;f$
\begin{align*}
&|c_i(t)|\leq C\left(\frac{\log|t|}{|t|}\right)^{1+\frac{\xs}{2\sqrt{2}}}\left|\left|\psi\right|\right|_{\mathcal{C}_\xF((s,-T)\times(0,\infty))}\\
&+C\left(\frac{\log|t|}{|t|}\right)^{\frac{1}{2}+\frac{\xs}{2\sqrt{2}}}\left(||h_1||_{\mathcal{C}_\xF((s,-T)\times(0,\infty))}
+||h_2||_{\mathcal{C}_\xF((s,-T)\times(0,\infty))}\right),\quad \forall\;t\in [s,-T],\;\;\;\forall\; i=1,...,k
\end{align*}
and

\begin{align*}
&\left|\frac{c_i(t)w'(r-\rho_i(t))}{\xF(t,r)}\right|\leq C\left(\frac{\log|t|}{|t|}\right)^{1-\frac{\xs}{2\sqrt{2}}}\left|\left|\psi\right|\right|_{\mathcal{C}_\xF((s,-T)\times(0,\infty))}\\
&+C\left(\frac{\log|t|}{|t|}\right)^{\frac{1}{2}-\frac{\xs}{2\sqrt{2}}}\left(||h_1||_{\mathcal{C}_\xF((s,-T)\times(0,\infty))}
+||h_2||_{\mathcal{C}_\xF((s,-T)\times(0,\infty))}\right),\quad \forall\;t\in [s,-T],\;\;\;\forall\; i=1,...,k.
\end{align*}
\label{cilemma}
\end{lemma}
\begin{proof}
For $i< j,$ we have
\begin{align}\nonumber
&\int_0^{\infty}r^{n-1}w'(r-\rho_i(t))w'(r-\rho_j(t))dr=\int_{-\rho_j(t)}^{\infty}(x+\rho_j(t))^{n-1}w'(x+(\rho_j(t)-\rho_i(t)))w'(x)dx\\ \nonumber
&\sum_{l=0}^{n-1}\binom{n-1}{l}(\rho_j(t))^{n-1-l}\int_{-\rho_j(t)}^{\infty}x^{l}w'(x+(\rho_j(t)-\rho_i(t)))w'(x)dx\\ \nonumber
&\leq (\rho_j^0(t))^{n-1}(\eta(t))^{n+1}\frac{|\log|t||}{t},
\end{align}
thus the system is nearly diagonal and we can solve it for $T$ big enough.

Also we can easily prove that
\begin{align}
\int_{0}^\infty x^{l}\xF(t,x)dx\leq C\sum_{j=1}^{k}(\rho_j^0(t)+\rho_{j-1}^0(t))^l\left(\frac{\log|t|}{|t|}\right)^\frac{\xs}{2\sqrt{2}}\label{****}
\end{align}
and
\begin{align}
\int_{0}^\infty x^{l}\xF(t,x)w'(r-\rho_j(t))dx \leq C\sum_{j=1}^{k}(\rho_j^0(t)+\rho_{j-1}^0(t))^l\left(\frac{\log|t|}{|t|}\right)^{\frac{1}{2}+\frac{\xs}{2\sqrt{2}}},\label{*****}
\end{align}
where $\rho_0^0=\rho_1-\eta,\;\rho_{k+1}^0=\infty.$

By assumptions on $\rho$ we have

\begin{align}
\left|\rho_j'(t)+\frac{n-1}{r+\rho_j(t)}\right|\leq C\frac{\log|t|}{|t|},\quad\text{if}\;\; \frac{\rho_1^0(t)+\rho_{0}^0(t)}{2}\leq r\leq \rho_k^0+\frac{\sqrt{2}+\xs}{\sqrt{2}-\xs}\eta,\label{xi'}
\end{align}
thus we can show
\be
\left|\int_{0}^\infty r^{n-1}\left(\rho_j'(t)+\frac{n-1}{r}\right)\psi w''(r-\rho_j(t))dr\right|\leq C\sum_{j=1}^{k}(\rho_j^0(t)+\rho_{j-1}^0(t))^{n-1}\left(\frac{\log|t|}{|t|}\right)^{1+\frac{\xs}{2\sqrt{2}}}.\label{estci1}
\ee

Now, by \eqref{*****}, we have

\begin{align}\nonumber
&\left|\int_{-\rho_j(t)}^\infty (r+\rho_j(t))^{n-1}\left(f'(w(r))-f'(z(t,r+\rho_j(t)))\right)\psi(t,r+\rho_j(t))w'(r)dx\right|\\ \nonumber
&\leq C\left|\left|\psi\right|\right|_{\mathcal{C}_\xF((s,-T)\times(0,\infty))}\frac{\log|t|}{|t|}\int_{-\rho_j(t)}^\infty(r+\rho_j(t))^{n-1}\xF(t,r+\rho_j(t))dr\\
&\leq C\left|\left|\psi\right|\right|_{\mathcal{C}_\xF((s,-T)\times(0,\infty))}\left(\frac{\log|t|}{|t|}\right)^{1+\frac{\xs}{2\sqrt{2}}}
\sum_{j=1}^{k}(\rho_j^0(t)+\rho_{j-1}^0(t))^{n-1}.
\label{estci}
\end{align}

Using all above and by simple calculations, we can reach at the proof of the first inequality of the Lemma.

The second inequality is a consequence of the fact that

\begin{align}\nonumber
\left|\frac{c_i(t)w'(x-\rho_i(t))}{\xF(t,x)}\right|\leq C\left(\frac{|t|}{\log|t|}\right)^{\frac{\xs}{\sqrt{2}}}.
\end{align}
The proof of Lemma is complete.
\end{proof}
\emph{Proof of Lemma \ref{cilemma*}.}

We will prove that there exists a unique solution of the problem (\ref{proci}) by using a fix point argument.

Let $$X^s=\{\psi:\;||\psi||_{\mathcal{C}_\xF((s,s+1)\times(0,\infty)}<\infty\}$$

We consider the operator $A^s: X^s\rightarrow X^s$ given by
$$A^s(\psi)=T^s(h-C(\psi)),$$
where $T^s(g)$ denotes the solution to (\ref{mainpro*}) and $C(\psi)=\sum_{j=1}^k c_i(t)w'(x-\rho_j(t)).$
Also by standard parabolic estimates we have
\be
||A^s(\psi)||_{\mathcal{C}_\xF((s,s+1)\times(0,\infty))}\leq C_0\left(
||h_2-C(\psi)||_{\mathcal{C}_\xF((s,s+1)\times(0,\infty))}+||h_1||_{\mathcal{C}_\xF((s,s+1)\times(0,\infty))}\right),\label{lem1}
\ee
for some uniform constant $C_0>0.$
We will show that the map $A^s$ defines a contraction mappping and we will apply the fixed point theorem to it.
To this end, set
$$c=C_0\left(||h_1||_{\mathcal{C}_\xF((s,-T)\times(0,\infty))}+||h_2||_{\mathcal{C}_\xF((s,-T)\times(0,\infty))}\right)$$
and
$$X^s_c=\{\psi:\;||\psi||_{C_\xF((s,s+1)\times(0,\infty))}<2c\},$$
where constant $C_0$ taken from  \eqref{lem1}, for $C(T,s)=C(s+1,s).$ We note here that by standard parabolic theory, the constant $C(T,s)= C_0|(-T-s)|.$

We claim that $A^s(X^s_c)\subset X^s_c,$ indeed by inequality \eqref{lem1} we have
\bea
\nonumber
&&||A^s(\psi)||_{\mathcal{C}_\xF((s,s+1)\times(0,\infty))}\leq C_0\left(
||h_2-C(\psi)||_{\mathcal{C}_\xF((s,s+1)\times(0,\infty))}+||h_1||_{\mathcal{C}_\xF((s,s+1)\times(0,\infty))}\right)\\ \nonumber
&\leq& C_0\left(||h_1||_{\mathcal{C}_\xF((s,-T)\times(0,\infty))}+||h_2||_{\mathcal{C}_\xF((s,-T)\times(0,\infty))}+
||C(\psi)||_{\mathcal{C}_\xF((s,s+1)\times(0,\infty))}\right)\\ \nonumber
&\leq&
\frac{C_0}{\sqrt{|s+1|}}\left(||\psi||_{\mathcal{C}_\xF((s,s+1)\times(0,\infty))}\right)+c\\ \nonumber
&\leq& c+c,
\eea
where in the above inequalities we have used Lemma \ref{cilemma} and we have chosen $|s|$ big enough. Next we show that $A^s$ defines a contraction map. Indeed, since $C(\psi)$ is linear in $\psi$ we have
\bea\nonumber
&&||A^s(\psi_1)-A^s(\psi_2)||_{\mathcal{C}_\xF((s,s+1)\times(0,\infty))}\\ \nonumber
&\leq&||C(\psi_1)-C(\psi_2)||_{\mathcal{C}_\xF((s,s+1)\times(0,\infty))}=
||C(\psi_1-\psi_2)||_{\mathcal{C}_\xF((s,s+1)\times(0,\infty))}\\ \nonumber
&\leq&\frac{C}{\sqrt{|s+1|}}||(\psi_1-\psi_2)||_{\mathcal{C}_\xF((s,s+1)\times(0,\infty))}\\ \nonumber
&\leq&\frac{1}{2}||(\psi_1-\psi_2)||_{\mathcal{C}_\xF((s,s+1)\times(0,\infty))}.
\eea
Combining all above, we have by fixed point theorem that there exist a $\psi^s\in X^s$ so that $A^s(\psi^s)=\psi^s,$ meaning that the equation (\ref{proci}) has a solution $\psi^s,$ for $-T=s+1.$

We claim that $\psi^s(t,x)$ can be extended to a solution on $(s,-T_0]\times(0,\infty),$ still satisfies the orthogonality condition (\ref{orthcond}) and the a priori estimate. To this end, assume that our solution $\psi(t,\cdot)$ exists for $s\leq t\leq -T,$ where $T>T_0$ is the maximal time of the existence. Since $\psi^s$ satisfies the orthogonality condition (\ref{orthcond}), we have by \eqref{estfix}
$$||\psi^s||_{\mathcal{C}_\xF((s,-T)\times(0,\infty))}\leq C\left(||h_2-C(\psi)||_{\mathcal{C}_\xF((s,-T)\times(0,\infty))}+||h_1||_{\mathcal{C}_\xF((s,-T)\times(0,\infty))}\right).$$
Thus if we choose $T_0$ big enough, we have by Lemma \ref{cilemma} that
\begin{align*}
||\psi^s||_{\mathcal{C}_\xF((s,-T)\times(0,\infty))}&\leq C\left(||h_1||_{\mathcal{C}_\xF((s,-T)\times(0,\infty))}+||h_2||_{\mathcal{C}_\xF((s,-T)\times(0,\infty))}\right)\\
&\leq C\left(||h_1||_{\mathcal{C}_\xF((s,-T_0)\times(0,\infty))}+||h_2||_{\mathcal{C}_\xF((s,-T_0)\times(0,\infty))}\right)
\end{align*}
It follows that $\psi^s$ can be extended past time $-T,$ unless $T=T_0.$ Moreover, (\ref{estfixmef}) is satisfied as well and $\psi^s$ also satisfies the orthogonality condition.\hfill$\Box$

\bigskip
\noindent{ \bf Proof of Proposition \ref{prop1} }
 Take a sequence $s_j\rightarrow-\infty$ and  $\psi_j=\psi^{s_j}$ where $\psi^{s_j}$ is the function   (\ref{proci}) with $s=s_j.$ Then by (\ref{estfix}), we can find a subsequence $\{\psi_j\}$ and $\psi$ such that $\psi_j\rightarrow\psi$ locally uniformly in $(-\infty,-T_0)\times(0,\infty).$

Using (\ref{estfix}) and standard parabolic theory we have that $\psi$ is a solution of (\ref{proci}) and satisfies (\ref{estfix**}).
The proof is concluded.

\section{The nonlinear problem}
Going back to the nonlinear problem, function $\psi$  is a solution of (\ref{mainpro}) if and only if  $\psi\in C_\xF((-\infty,-T_0)\times(0,\infty))$ solves the fixed point problem
\be
  \psi= A(\psi ) \label{2.14}
\ee
where
$$
A(\psi ) := T(\overline{E}(\psi)),
$$
$A$ is the operator in Proposition \ref{fixth} and
$$\overline{E}(\psi)=E+N(\psi)-\sum_{i=1}^k c_i(t)w'(x-\rho_i(t)).$$

Let $T_0>1,$ we define
$$\xL=\left\{h\in C^1(-\infty,-T_0]:\;\sup_{t\leq -T_0}|h(t)|+\sup_{t\leq -T_0}\left(\frac{|t|}{\log|t|}|h'(t)|\right)<1\right\}$$
and
$$||h||_\xL=\sup_{t\leq -T_0}(|h(t)|)+\sup_{t\leq -T_0}\left(\frac{|t|}{\log|t|}|h'(t)|\right).$$

The main goal in this section is to prove the following Proposition.
\begin{prop}
Let $\xs<\sqrt{2}$ and $\xn=\frac{\sqrt{2}-\xs}{2\sqrt{2}}$. There exists number $T_0> 0,$ depending only on $\xs$ such that for any given functions $h$ in $\xL,$ there is a solution  $ \psi= 	 \Psi(h)$ of (\ref{2.14}), with respect $\rho=\rho^0+h.$ The solution $\psi$
satisfies the orthogonality conditions (2.9)-(2.10).
Moreover, the following estimate holds
\be
||\Psi(h_1)-\Psi(h_2)||_{\mathcal{C}_\xF((-\infty,-T_0)\times(0,\infty))}\leq C\left(\frac{\log T_0}{T_0}\right)^\xn||h_1-h_2||_\xL,\label{diafora}
\ee
where $C$ is a universal constant.\label{mainproposition}
\end{prop}

To prove Proposition \ref{mainproposition} we need to prove some lemmas first.

Set
$$X_{T_0}=\{\psi:\;||\psi||_{\mathcal{C}_\xF((-\infty,-T_0)\times(0,\infty))}<2C_0\left(\frac{\log T_0}{T_0}\right)^\xn\},$$
for some fixed constant $C_0.$

We denote by $N(\psi,h)$ the function $N(\psi)$ in \eqref{mainlem} with respect $\psi$ and $\rho=\rho^0+h.$
 Also we denote by $z_i$ the respective function in \eqref{z} with respect $\rho=\rho_i=\rho^0+h_i,$ $i=1,2.$
\begin{lemma}
Let $h_1,\;h_2\in \xL$ and $\psi_1,\;\psi_2\in X_{T_0}.$
Then there exists a constant $C=C(C_0)$ such that
\begin{align*}
||N(\psi_1,h_1)&-N(\psi_2,h_2)||_{\mathcal{C}_\xF((-\infty,-T_0)\times(0,\infty))}\\ &
\leq C\left(\frac{\log T_0}{T_0}\right)^\xn\left(||\psi_1-\psi_2||_{\mathcal{C}_\xF((-\infty,-T_0)\times(0,\infty))}+||h_1-h_2||_{\xL}\right)
\end{align*}\label{dia1}
\end{lemma}
\begin{proof}
First we will prove that there exists constant $C>0$ which depends only on $C_0$ such that
\be
||N(\psi_1,h_1)-N(\psi_2,h_1)||_{\mathcal{C}_\xF((-\infty,-T_0)\times(0,\infty))}\leq C\left(\frac{\log T_0}{T_0}\right)^\xn||\psi_1-\psi_2||_{\mathcal{C}_\xF((-\infty,-T_0)\times(0,\infty))}.\label{n1}
\ee
By straight forward calculation we can easily show that
$$|N(\psi_1,h_1)-N(\psi_2,h_1)|\leq C\left(\frac{\log T_0}{T_0}\right)^\xn|\psi_1-\psi_2|(\xF+\xF^2),$$
where the constant $C$ depend on $C_0$ and the proof of \eqref{n1} follows.

Now we will prove that
\be
||N(\psi_2,h_1)-N(\psi_2,h_2)||_{\mathcal{C}_\xF((-\infty,-T_0)\times(0,\infty))}\leq C\left(\frac{\log T_0}{T_0}\right)^\xn||h_1-h_2||_{\xL}.\label{n2}
\ee
where the constant $C$ depends on $C_0.$

By straightforward calculations we have
\begin{align}\nonumber
|N(\psi_2,h_1)-N(\psi_2,h_2)|&=|-(z_1+\psi_2)^3+z_1^3+3z_1^2\psi_2+(z_2+\psi_2)^3-z_2^3-3z_2^2\psi_2|\\
&\leq C\left(\frac{\log T_0}{T_0}\right)^\xn|h_1-h_2|\xF^2,\label{n2*}
\end{align}
which implies \eqref{n2}.
By \eqref{n1} and \eqref{n2} the result follows.
\end{proof}
We denote by $E(h)$ the function $E$ in \eqref{mainlem} with respect $\psi$ and $\rho=\rho^0+h.$
\begin{lemma}
Let $h_1,\;h_2\in \xL.$
Then there exists constant $C=C(C_0)$ such that
\be
||E(h_1)-E(h_2)||_{\mathcal{C}_\xF((-\infty,-T_0)\times(0,\infty))}\leq C\left(\frac{\log T_0}{T_0}\right)^\xn||h_1-h_2||_{\xL}
\ee\label{dia2}
\end{lemma}
\begin{proof}
Set $\rho=\rho^0+h_1,$ $\xz=\rho^0+h_2.$ In view of the proof of Lemma \ref{remark} and the above inequality we have
\begin{align*}
|f(z_1(t,r))&-\sum_{j=1}^{k}(-1)^{j+1}f(w(r-\rho_j))-f(z_2(t,r))+\sum_{j=1}^{k}(-1)^{j+1}f(w(r-\xz_j))\\
&\leq C|h_1-h_2||w'(r-\rho_{j-1}^0(t))|,\;\;\text{if}\;\; \frac{\rho_j^0(t)+\rho_{j-1}^0(t)}{2}\leq r\leq \frac{\rho_j^0(t)+\rho_{j+1}^0(t)}{2},
\end{align*}
with $\rho_0^0=\rho_1^0-\eta$ and $\rho_{k+1}^0=\infty.$

By the assumptions on $\xz$ we have that there exists a positive constant $C=C(N,k,\xs)>0$ such that
$$|\xz_j'(t)+\frac{n-1}{r}|\leq C\left(\frac{\log |t|}{|t|}\right),\quad\text{if}\;\; \frac{\rho_1^0(t)+\rho_{0}^0(t)}{2}\leq r\leq \rho_k^0+\frac{\sqrt{2}+\xs}{\sqrt{2}-\xs}\eta,$$

$$\frac{w'(r-\rho_j(t))}{\xF}\leq C\left(\frac{\log |t|}{|t|}\right),\quad\forall\;\; r\geq \rho_k^0+\frac{\sqrt{2}+\xs}{\sqrt{2}-\xs}\eta, $$
$$\frac{w'(r-\rho_j(t))}{\xF}\leq C\left(\frac{\log |t|}{|t|}\right)^\xn,\quad\forall\;\; r\leq \frac{\rho_1^0(t)+\rho_{0}^0(t)}{2},\;j=1,...,k. $$
and
\begin{align*}
&\frac{1}{r}|\sum_{j=1}^{k}(-1)^{j+1}w'(r-\rho_j(t))-\sum_{j=1}^{k}(-1)^{j+1}w'(r-\xz_j(t))|\\
&\leq \frac{C}{r}|w'(r-\rho_j^0(t))|||h_1-h_2||_{\xL},\quad\forall r\leq\frac{\rho_1^0(t)+\rho_{0}^0(t)}{2}
\end{align*}
Combining all above we can reach to the desired result by simple arguments.
\end{proof}

\begin{lemma}
Let $h_1,\;h_2\in \xL,$ $\psi_1,\;\psi_2,\;\psi\in X.$  Also let $C(\psi,h,t)=(c_1(t),...,c_k(t))$ satisfy \eqref{ci(t)} with respect $\psi$ and $\rho=\rho^0+h.$ Then
\begin{align}\nonumber
|C(\psi_1,h_1,t)-C(\psi_2,h_2,t)|&\leq C\left(\frac{|\log|t|}{|t|}\right)^{1+\frac{\xs}{2\sqrt{2}}}||\psi_1-\psi_2||_{\mathcal{C}_\xF((-\infty,-T_0)\times(0,\infty))}\\ &+C\left(\frac{\log|t|}{|t|}\right)^{\xn+\frac{\xs}{\sqrt{2}}}||h_1-h_2||_{\xL},
\end{align}
for some positive constant $C_0$ which depend only on $C_0.$
\label{dia3}
\end{lemma}
\begin{proof}
For the proof of Lemma, we do very similar calculations like in Lemmas \ref{cilemma}, \ref{dia1}, \ref{dia2} and we omit it.
\end{proof}
\medskip
\emph{Proof of Proposition \ref{mainproposition} }
a)
We consider the operator $A: \mathcal{C}_\xF((-\infty,-T_0)\times(0,\infty))\rightarrow \mathcal{C}_\xF((-\infty,-T_0)\times(0,\infty)),$
where $A(\psi)$ denotes the solution to (\ref{2.14}).
We will show that the map $A$ defines a contraction mapping and we will apply the fixed point theorem to it.
First we note by Lemma \ref{remark} and Theorem \ref{fixth} that
$$
||A(0)||_{\mathcal{C}_\xF((-\infty,-T_0)\times(0,\infty))}\leq  C_0\left(\frac{\log T_0}{T_0}\right)^\xn.
$$
and by Proposition \ref{fixth} and Lemma \ref{cilemma*}
\bea\nonumber
&&||A(\psi_1)-A(\psi_2)||_{\mathcal{C}_\xF((-\infty,-T_0)\times(0,\infty))}\\ \nonumber
&\leq&  C\left(\frac{\log T_0}{T_0}\right)^\xn
\left(||\psi_1-\psi_2||_{\mathcal{C}_\xF((-\infty,-T_0)\times(0,\infty))}\right)
\eea
providing
$$
||\psi_i||_{\mathcal{C}_\xF((-\infty,-T_0)\times(0,\infty))}\leq2 C_0\left(\frac{\log T_0}{T_0}\right)^\xn.
$$
Thus if we choose $T_0$ big enough we can apply the fix point Theorem in
$$X_{T_0}=\{\psi:\;||\psi||_{\mathcal{C}_\xF((-\infty,-T_0)\times(0,\infty))}<2 C_0\left(\frac{\log T_0}{T_0}\right)^\xn\},$$
to obtain that there exists $\psi$ such that $A(\psi)=\psi.$

b) For simplicity we set  $\psi^1 = 	\Psi(h_1)$ and
 $\psi^2 = 	\Psi(h_2).$ The estimate will be obtained by applying the estimate \eqref{estfix}. However, because each
 $\psi^i$ satisfies the orthogonality conditions (\ref{orthcond})  with $\rho(t) = \rho^i(t) := \rho^0(t)+ h_i(t),$ the
difference  $\psi^1-\psi^2$ doesn't satisfy an exact orthogonality condition. To overcome this technical difficulty
we will consider instead the difference $Y :=  \psi^1-\overline{\psi}^2,$ where

$$\overline{\psi}^2=\psi^2-\sum_{i=1}^k\xl_i(t)w'(x-\rho_i^1).$$
with
$$
\sum_{i=1}^k\xl_i(t)\int_{0}^\infty r^{n-1}w'(r-\rho_i^1(t))w'(r-\rho_j^1(t))dx=\int_{0}^\infty r^{n-1}\psi^2(t,r)w'(x-\rho_j^1(t)) dr,
$$
$j=1,...,k.$
Clearly, $Y$ satisfies the orthogonality conditions (\ref{orthcond}) with $\rho(t) = \rho^1(t).$ Denote by $L^i_t$ the
operator
$$
L^i_t\psi^i=\psi^i_t-\psi^i_{rr}-\frac{n-1}{r}\psi^i_r-f'(z^i(t,x))\psi^i.
$$
By Lemmas \ref{dia1}, \ref{dia2} and \ref{dia3}
and the fact that
$$\frac{w(r-\rho_j^0(t))}{\xF}\leq C|t|^{\frac{\xs}{\sqrt{2}}},\quad \forall r>0\;\text{and}\;j=1,...,k,$$
we can easily prove
\begin{align}\nonumber
||Y||_{\mathcal{C}_\xF((-\infty,-T_0)\times(0,\infty))}&\leq C \left(\frac{\log T_0}{T_0}\right)^\xn\left(||\psi_1-\psi_2||_{\mathcal{C}_\xF((-\infty,-T_0)\times(0,\infty))}+||h_1-h_2||_{\xL}\right) \\ &+C\left(\frac{\log T_0}{T_0}\right)^\xn\left(\sum_{i=1}^k\sup_{t\in(-\infty,-T_0)}|t|^{\frac{\xs}{\sqrt{2}}}|\xl_i(t)|\right).\label{pao0}
\end{align}
Now, by orthogonality conditions \eqref{orthcond} and \eqref{****}, we have
\bea\nonumber
\left|\int_{0}^\infty r^{n-1}\psi^2(t,x)w'(r-\rho_j^1(t))dr\right|&=&\left|\int_{0}^\infty r^{n-1}\psi^2(t,r)(w'(r-\rho_j^1(t))-w'(r-\rho_j^2(t)))dr\right|\\
&\leq& C \left(\frac{\log T_0}{T_0}\right)^\xn|t|^{-\frac{\xs}{\sqrt{2}}}||h_1-h_2||_{\xL}\sum_{i=1}^k(\rho_i^0)^{n-1}.\label{pao1}
\eea

Now
\begin{align}\nonumber
&\left|\frac{d\int_{0}^\infty r^{n-1}\psi^2(t,r)w'(r-\rho_j^1(t))dr}{dt}\right|\\
&=\left|\frac{d\int_{0}^\infty r^{n-1}\psi^2(t,r)(w'(r-\rho_j^1(t))-w'(r-\rho_j^2(t)))dr}{dt}\right|.\label{pao2}
\end{align}
But
\begin{align*}
&\int_{0}^\infty r^{n-1}\psi^2_t(t,r)(w'(r-\rho_j^1(t))-w'(r-\rho_j^2(t)))dr\\
&=-\int_{0}^\infty r^{n-1}\psi^2_{r}(t,x)(w''(r-\rho_j^1(t))-w''(r-\rho_j^2(t)))dr\\
&+\int_{0}^\infty r^{n-1}L^2_t\psi^2(w'(r-\rho_j^1(t))-w'(r-\rho_j^2(t)))dr\\
&+\int_{(0,\infty)}f'(z^2(t,x))\psi^2(t,x)(w'(r-\rho_j^1(t))-w'(r-\rho_j^2(t)))dr\\
&=\int_{0}^\infty r^{n-1}\psi^2(t,r)(r^{n-1}(w''(r-\rho_j^1(t))-w''(r-\rho_j^2(t))))_rdr\\
&+\int_{0}^\infty r^{n-1}L^2_t\psi^2(w'(r-\rho_j^1(t))-w'(r-\rho_j^2(t)))dr\\
&+\int_{0}^\infty r^{n-1}f'(z^2(t,r))\psi^2(t,r)(w'(r-\rho_j^1(t))-w'(r-\rho_j^2(t)))dr.
\end{align*}

By the fix point argument in a) we have that
\begin{align}\nonumber
\left|\int_{0}^\infty r^{n-1}\psi^2_t(t,r)(w'(r-\rho_j^1(t))-w'(r-\rho_j^2(t)))dr\right|\\
\leq C \left(\frac{\log T_0}{T_0}\right)^\xn|t|^{-\frac{\xs}{\sqrt{2}}}||h_1-h_2||_{\xL}\sum_{i=1}^k(\rho_i^0)^{n-1}.\label{pao3}
\end{align}
By \eqref{pao1}, \eqref{pao2}, \eqref{pao3} and definitions of $\xl_i$ we have that
$$
 |\xl_i(t)|+|\xl_i'(t)|\leq C \left(\frac{\log T_0}{T_0}\right)^\xn|t|^{-\frac{\xs}{\sqrt{2}}}||h_1-h_2||_{\xL}
$$
Combining all above we have that
$$
||Y||_{\mathcal{C}_\xF((-\infty,-T_0)\times(0,\infty))}\leq C \left(\left(\frac{\log T_0}{T_0}\right)^\xn||\psi_1-\psi_2||_{\mathcal{C}_\xF((-\infty,-T_0)\times(0,\infty))}+||h_1-h_2||_{\xL}\right)
$$
But
\begin{align*}
||\psi_1-\psi_2||_{\mathcal{C}_\xF((-\infty,-T_0)\times(0,\infty))}&\leq ||Y||_{\mathcal{C}_\xF((-\infty,-T_0)\times(0,\infty))}+C\left(\sum_{i=1}^k\sup_{t\in(-\infty,-T_0)}|t|^{\frac{\xs}{\sqrt{2}}}|\xl_i(t)|\right)\\
&\leq  C \left(\frac{\log T_0}{T_0}\right)^\xn\left(||\psi_1-\psi_2||_{\mathcal{C}_\xF((-\infty,-T_0)\times(0,\infty))}+||h_1-h_2||_{\xL}\right),
\end{align*}
and the proof of inequality \eqref{diafora} follows if we choose $T_0$ big enough.\hfill$\Box$

\setcounter{equation}{0}
\section{the choice of $\rho_i$}\label{xiint}
Let $T_0$ big enough, $\frac{\sqrt{2}}{2}<\xs<\sqrt{2}$ and $\psi\in\mathcal{C}_\xF((-\infty,-T_0)\times(0,\infty))$ be the solution of the problem (\ref{mainpro}).
We want to find $\rho_i$ such that $c_i=0$ in \eqref{ole} for any $i=1,...,k.$

We will study only the error term $E.$
Let $1<j<k,$ then we have that
\begin{align}\nonumber
&\int_0^\infty r^{n-1}\left(f(z(t,r))-\sum_{i=1}^{k}(-1)^{i+1}f(w(r-\rho_i(t)))\right)w'(r-\rho_j(t))dr\\ \nonumber
&=\int_{-\rho_j(t)}^\infty (x+\rho_j(t))^{n-1}\left(f(z(t,x+\rho_j(t)))-\sum_{i=1}^{k}(-1)^{i+1}f(w(x+\rho_j(t)-\rho_i(t)))\right)w'(x)dx.
\end{align}
For simplicity we assume that $i$ is even.
Set
\begin{align*}
g&=\sum_{i=1}^{j-2}(-1)^{i+1}\left(w(x+\rho_j(t)-\rho_{i}(t))-1\right)
\\&+\sum_{i=j+2}^{k}(-1)^{i+1}\left(w(x+\rho_j(t)-\rho_{i}(t))+1\right),
\end{align*}
$$g_1=w(x+\rho_j-\rho_{j-1})-1,$$
and
$$g_2=w(x+\rho_j-\rho_{j+1})+1$$

By straightforward calculations we have

\begin{align}\nonumber
&\int_{-\rho_j(t)}^\infty (x+\rho_j(t))^{n-1}\left(f(z(t,x-\rho_j(t)))-\sum_{i=1}^{k}(-1)^{i+1}f(w(x+\rho_j(t)-\rho_i(t)))\right)w'(x)dx\\ \nonumber
&=3\int_{-\rho_j(t)}^\infty (x+\rho_j(t))^{n-1}(g_1+g_2)(1-w^2(x))w'(x)dx+3\int_{-\rho_j(t)}^\infty g_1^2(1+w(x))w'(x)dx \\ \nonumber
&+3\int_{-\rho_j(t)}^\infty (x+\rho_j(t))^{n-1}g_2^2(w(x)-1)w'(x)dx
+\int_{-\rho_j(t)}^\infty (x+\rho_j(t))^{n-1}F_0(t,x)w'(x)dx,
\end{align}
where $$F_0(t,x)=O(g)+O(g_1g_2).$$

By a simple argument we can show

\begin{align*}
&\int_{-\rho_j(t)}^\infty (x+\rho_j(t))^{n-1}g_1(1-w^2(x))w'(x)dx\\
&=-2e^{-\sqrt{2}(\rho_j-\rho_{j-1})}\sum_{l=1}^{n-1}\binom{n-1}{l}\rho_j^{l}\int_{\mathbb{R}}x^{{n-1}-l}e^{-\sqrt{2}x}(1-w^2(x))w'(x)dx\\
&+F_2(\rho)
\end{align*}
where $F_2$ satisfies
\begin{align*}
|F_2|&\leq C\sum_{l=1}^k|\rho_j(t)-\rho_{j-1}(t)|^l e^{-\sqrt{2}(\rho_j-\rho_{j-1})}\sum_{l=1}^{n-1}\rho_j^l(t)+O(\rho_j^{n-1}e^{-\sqrt{2}\rho_j(t)}).
\end{align*}

Similarly for  $g_2,$ $j=1,...,k,$ and in view of the proof of Lemma \ref{cilemma} we can reach at the ODE, for $\rho=(\rho_1,...,\rho_k)$

\be
\rho_j'+\frac{n-1}{\rho_j}-\xb e^{-\sqrt{2}(\rho_{j+1}-\rho_j)}+\xb e^{-\sqrt{2}(\rho_{j}-\rho_{j-1})}=F_i(\rho',\rho),\qquad j=1,2,...,k,\;\;t\in(-\infty,-T_0],\label{ode*}
\ee
with $\rho_{k+1}=\infty,$ $\rho_0=-\infty$ and
\be \label{beta} \xb=\frac{6\int_{\mathbb{R}}e^\frac{2x}{\sqrt{2}}(1-w^2(x))w'(x)dx}{\int_{\mathbb{R}}(w'(x))^2dx}.\ee

 We recall here that, we assume $T_0>1$ and we denote by
$$\xL=\{h\in C^1(-\infty,-T_0]:\;\sup_{t\leq -T_0}|h(t)|+\sup_{t\leq -T_0}\frac{|t|}{\log|t|}|h'(t)|<1\}$$
and
$$||h||_\xL=\sup_{t\leq -T_0}(|h(t)|)+\sup_{t\leq -T_0}(\frac{|t|}{\log|t|}|h'(t)|).$$
We set
$$\overline{\mathbf{F}}(h',h)=\mathbf{F}(\rho',\rho),$$
where $\rho=\rho^0+h.$

Working like above and Lemmas \ref{pao1}, \ref{pao2}, \ref{pao3} and using \eqref{diafora} we have the following result.

\begin{prop}
Let $\frac{\sqrt{2}}{2}<\xs<\sqrt{2}$ and $h,\;h_1,\;h_2\in \xL.$ Then there exists a constant $C=C(\xs,n,k)>0$ such that
$$|\overline{\mathbf{F}}(h',h)|\leq \frac{C}{|t|},$$
and
$$|\overline{\mathbf{F}}(h'_1,h_2)-\overline{\mathbf{F}}(h'_1,h_2)|\leq C\left(\frac{\log|t|}{|t|}\right)^{\frac{1}{2}+\frac{\xs}{\sqrt{2}}}||h_1-h_2||_{\xL}.$$\label{remark2}
\end{prop}

In the rest of this section we will study the system \ref{ode*} using some ideas in \cite{del pino}.
\subsection{the choice of $\rho^0$}\label{ksi0}
\begin{lemma}
There exists a unique solution $\eta$ with $\eta(-1)=0$ of the problem
\be
\eta'+\frac{1}{2t}\eta+ e^{-\sqrt{2}\eta}=0,\quad t\in(-\infty,-1].\label{odesimp2}
\ee

Furthermore there exist $\tilde{T}_0$ and a positive constant $C=C(\tilde{T}_0)>0$ such that

\begin{align}
-\frac{1}{\sqrt{2}}\log\left(C^{-1}\frac{\log |t|}{|t|}\right)&\leq\eta(t)\leq -\frac{1}{\sqrt{2}}\log\left(C\frac{\log |t|}{|t|}\right),\quad\forall t\leq-\tilde{T}_0,
\\
0&\leq-\eta'(t)\leq C\frac{\log |t|}{|t|},\quad\forall t\leq-\tilde{T}_0.
\end{align}\label{lemodesimpl1}
\end{lemma}
\begin{proof}
By standard ODE theory we have that there exists a unique solution $\eta$ of the problem \ref{odesimp2} which satisfies

\be
\eta=\frac{1}{\sqrt{-t}}\int_{t}^{-1}\sqrt{-t}e^{-\sqrt{2}\eta(s)}ds,\quad t\leq -1,\label{odesimp3}
\ee

Note that  by \eqref{odesimp3}, $ \eta\geq0$ and $\eta$ is not bounded.

Next we claim that $\eta$ is non increasing. We will prove it by contradiction, we assume that $\eta'$ changes signs.

First we note that, since $\eta(t)>0\;\forall\;t\leq-1,$ $\eta(-1)=0$ and $\eta$ is not bounded, we can assume that there exist $t_0> t_1$ such that $\eta'(t_0)=\eta'(t_1)=0$ and
$$\eta'(t)<0,\quad \forall t\in (t_0,-1)\quad\text{and}\quad \eta'(t)>0,\quad \forall t\in (t_1,t_0).$$

But by \eqref{odesimp2}, we have that

$$\frac{1}{-2t_1}\eta(t_1)<\frac{1}{-2t_0}\eta(t_0)= e^{-\sqrt{2}\eta(t_0)}< e^{-\sqrt{2}\eta(t_1)}=\frac{1}{-2t_1}\eta(t_1),$$
which is clearly a contradiction.

Now since $\eta\geq0$ we have by \eqref{odesimp2}
\begin{align}
\eta'(t)\geq -e^{-\sqrt{2}t}\Rightarrow \left(e^{\sqrt{2}\eta}\right)'&\geq -\sqrt{2}\Rightarrow\eta(t)\leq -\frac{1}{\sqrt{2}}\log(-\sqrt{2}(t+1)),\quad\forall t\leq-1.\label{12}
\end{align}

Using the fact that $\eta$ is non increasing, \eqref{odesimp3} and \eqref{12} we have
$$e^{-\sqrt{2}\eta(t)}\frac{1}{\sqrt{-t}}\int_{t}^{-1}\sqrt{-t}ds\leq\frac{1}{\sqrt{-t}}\int_{t}^{-1}\sqrt{-t}e^{-\sqrt{2}\eta(s)}ds=\eta\leq -\frac{1}{\sqrt{2}}\log(-\sqrt{2}(t+1)),$$
which implies the existence of $C=C(\tilde{T}_0,n)>0$ such that
$$e^{-\sqrt{2}\eta(t)}\leq -C\frac{\log(-\sqrt{2}(t+1))}{t}\quad\text{and}\quad \eta(t)\geq-\frac{1}{\sqrt{2}}\log\left(C\frac{\log(-\sqrt{2}(t+1))}{-t}\right),\quad\forall t\leq-\tilde{T}_0.$$

By \eqref{12} and the above inequality we can easily obtain that there exists $C_1=C_1(\tilde{T_0},n)>0$ such that
$$\eta(t)\geq C_1\log(-t),\quad\forall t\leq-\tilde{T}_0.$$

Now, using the fact that $\eta$ is non increasing, \eqref{odesimp2} and the above inequality, we have
$$ e^{-\sqrt{2}\eta(t)}\geq C_2\log\left(\frac{\log(-t)}{-t}\right),\quad\forall t\leq-\tilde{T}_0,$$
where $C_2=C_2(\tilde{T_0},n)>0$ and the result follows.
\end{proof}

\begin{lemma}
Let $$b_l=-\frac{1}{\sqrt{2}}\log\left(\frac{1}{2\xb}(k-l)l\right),\qquad l=1,...,k-1,$$
and $$-\xg_j=\xg_{k-j+1}=\frac{1}{2}\sum_{i=j}^{k-j}b_i,\qquad\mathrm{for}\;j\leq\frac{k}{2}.$$

Then the function
$\tilde{\rho}_j^0(t)=\left(j-\frac{k+1}{2}\right)\eta+\xg_j$ is a solution of

\be
\rho_j'+\frac{1}{2t}\rho_j-\xb e^{-\sqrt{2}(\rho_{j+1}-\rho_j)}+\xb e^{-\sqrt{2}(\rho_{j}-\rho_{j-1})}=\frac{1}{2t}\xg_j,\qquad j=1,2,...,k,\;\;t\in(-\infty,-1],\label{ode}
\ee
with $\rho_{k+1}=\infty$ and $\rho_0=-\infty$  and  $\eta$ is the function in Lemma \ref{lemodesimpl1}.\label{lemodesimpl2}
\end{lemma}
\begin{proof}
We set
$$R_l(\rho):=-e^{-\sqrt{2}(\rho_{j+1}-\rho_j)}+e^{-\sqrt{2}(\rho_{j}-\rho_{j-1})},$$

$$
\textbf{R}(\rho)=\left[ \begin{array}{ccc}
R_1(\rho) \\
\vdots  \\
R_k(\rho)  \end{array} \right]
$$
and
$$\mathbf{\xg}=[\xg_1,...,\xg_k]^T\quad\text{and}\quad\mathbf{b}=[b_1,...,b_{k-1}]^T.$$
We want to solve the system $\rho'+\frac{1}{2t}\rho+\xb\textbf{R}(\rho)=\frac{1}{2t}\mathbf{\xg}.$ To do so we find first a convenient representation of the operator $\textbf{R}(\rho).$ Let us consider the auxiliary variables

$$
\textbf{v}:=\left[ \begin{array}{ccc}
\mathbf{\overline{v}} \\
v_k
  \end{array} \right],
  \qquad
\mathbf{\overline{v}}= \left[ \begin{array}{ccc}
v_1 \\
\vdots  \\
v_{k-1}  \end{array} \right],
$$

defined in terms of $\rho$ as
$$v_l=\rho_{l+1}-\rho_l\;\;\;\mathrm{with}\;l=1,...,k-1,\qquad v_k=\sum_{l=1}^k\rho_l,$$
and define the operators
$$
\mathbf{S}(\textbf{v}):=\left[ \begin{array}{ccc}
\overline{\mathbf{S}}(\mathbf{\overline{v}}) \\
0
  \end{array} \right],
  \qquad
\overline{\mathbf{S}}(\mathbf{\overline{v}})= \left[ \begin{array}{ccc}
S_1(\overline{\mathbf{v}}_1) \\
\vdots  \\
S_{k-1}(\overline{\mathbf{v}}_1)  \end{array} \right],
$$
where $S_l(\overline{\mathbf{v}}):R_{l+1}(\rho)-R_l(\rho)=$
$$
\Bigg\{ \begin{array}{ccc}
2e^{-\sqrt{2}v_1}-e^{\sqrt{2}v_2} &\mathrm{if}\qquad l=1  \\
-e^{\sqrt{2}v_{l+1}}+2e^{-\sqrt{2}v_l}-e^{\sqrt{2}v_{l-1}}&\mathrm{if}\qquad 2\leq l\leq k-2  \\
2e^{-\sqrt{2}v_k}-e^{\sqrt{2}v_{k-1}} &\mathrm{if}\qquad l=k-1
  \end{array}
$$

Then the operators  $\mathbf{R}$ and $\mathbf{S}$ are in correspondence through the formula
$$\mathbf{S}(\mathbf{v})=\mathbf{B}\mathbf{R}(\mathbf{B}^{-1}\mathbf{v}),$$
where $\mathbf{B}$ is the constant, invertible $k\times k$ matrix
$$\mathbf{\mathbf{B}}=\left[ \begin{array}{ccccc}
-1 & 1 & 0&\cdots&0 \\
0 & -1 & 1&\cdots&0 \\
\vdots & \ddots & \ddots&\ddots&\vdots\\
0&\cdots&0&-1&1\\
1&\ldots& 1&1&1
\end{array} \right]$$

and then through the relation $\rho=\mathbf{B}^{-1}\mathbf{v}$ the system $\rho'+\frac{1}{2t}\rho+\xb\textbf{R}(\rho)=\frac{1}{2t}\mathbf{\xg}$ is equivalent to $\mathbf{v}'+\frac{1}{2t}\mathbf{v}+\xb\mathbf{S}(\mathbf{v})=\frac{1}{2t}\mathbf{b},$ which decouples into

\bea
\label{18}
\overline{\mathbf{v}}'+\frac{1}{2t}\overline{\mathbf{v}}+\xb\overline{\mathbf{S}}(\mathbf{\overline{v}})&=&\frac{1}{2t}\mathbf{b},\\ \nonumber
v_k'&=&0,
\eea
where
\be
\overline{\mathbf{S}}(\mathbf{\overline{v}})= \mathbf{C}\left[ \begin{array}{ccc}
e^{-\sqrt{2} v_1} \\
\vdots  \\
e^{-\sqrt{2}v_{k-1}}  \end{array} \right]
,\qquad
\mathbf{C}=\left[ \begin{array}{ccccc}
2 & -1 & 0&\cdots&0 \\
-1 & 2 & -1&\cdots&0 \\
\vdots & \ddots & \ddots&\ddots&\vdots\\
0&\cdots&-1&2&-1\\
0&\ldots& &-1&2
\end{array} \right].\label{C}
\ee
We claim now that the function
\be
\overline{v}_l^0(t)=\eta+b_l.\label{v}
\ee
is a solution of \eqref{18}.

Indeed, substituting this expression into the system we see that the following equations for the numbers $b_l$ are satisfied
$$
\mathbf{C}\left[ \begin{array}{ccc}
e^{-\sqrt{2} b_1} \\
\vdots  \\
e^{-\sqrt{2}b_{k-1}}  \end{array} \right]=\frac{1}{\xb}\left[ \begin{array}{ccc}
1 \\
\vdots  \\
1  \end{array} \right]
$$

Now we note that
$b_l=b_{k-l}$ for $l=1,..,k-1,$ thus by (\ref{ode}) we have that $$\rho_{k-j+1}=-\rho_{j},\;\;j\leq\frac{k}{2},$$
and
$$\rho_j=\frac{1}{\sqrt{2}}\left(j-\frac{k+1}{2}\right)\eta+\xg_j.$$
and the result follows

\end{proof}
\subsection{the solution of the problem \eqref{ode*}}\label{general}
We keep the notations of the previous subsection. Set $\xz=\sqrt{-2(n-1)t},$
and $\mathbf{e}=[1,...,1]^T.$

 We look for solutions of the form $\rho=\sqrt{-2(n-1)t}\mathbf{e}+\tilde{h},$
then $\tilde{h}$ satisfies

$$\tilde{h}'+\frac{1}{2t}\tilde{h}=\mathbf{F}(\tilde{h}'+\mathbf{e}\xz',\tilde{h}+\mathbf{e}\xz)+\frac{n-1}{\xz}\mathbf{e}+\frac{1}{2t}\tilde{h}-\xb\mathbf{R}(\tilde{h}),\quad \mathrm{in}\;(-\infty,-T_0]
$$
where $T_0\geq \tilde{T}_0.$

Let $\eta$ be the function in Lemma \ref{lemodesimpl1}, we look for solutions of the form $\tilde{h}=\tilde{\rho}^0(t)+h$ then $h$ satisfies
\begin{align}\nonumber
h'+\frac{1}{2t}h+\xb D_\rho\mathbf{R}(\tilde{\rho}^0(t))h&=\mathbf{F}(h'+\mathbf{e}\xz'+(\tilde{\rho}^0)',\tilde{h}+\mathbf{e}\xz+\tilde{\rho}^0)
+\frac{n-1}{\xz}\mathbf{e}+\frac{1}{2t}(h+\tilde{\rho}^0(t))\\ \nonumber &-\xb\mathbf{R}(\tilde{\rho}^0(t)+h)+\xb\mathbf{R}(\tilde{\rho}^0(t))+\xb D_\rho\mathbf{R}(\tilde{\rho}^0(t))h+\frac{1}{2t}\mathbf{\xg}\\
&:=\mathbf{E}(h',h)+\frac{1}{2t}\mathbf{\xg},\quad \mathrm{in} \;(-\infty,-T_0]\label{ode1}
\end{align}
where $\mathbf{\xg}=[\xg_1,...,\xg_2]^T.$

Set $v^0=\mathbf{B}\eta$ and $p=\mathbf{B}h.$ Then we have that $\mathbf{E}(h',h)=\mathbf{E}(\mathbf{B}^{-1}h',\mathbf{B}^{-1}h)=\mathbf{E}(p',p),$ and by
$\mathbf{S}(\mathbf{v})=\mathbf{B}\mathbf{R}(\mathbf{B}^{-1}\mathbf{v}),$
we have that
$\mathbf{S}(\mathbf{v^0})=\mathbf{B}\mathbf{R}(\tilde{\rho}^0(t))\mathbf{B}^{-1}.$

Thus (\ref{ode1}) is equivalent to
\be
p'+\frac{1}{2t}p+\xb D_v\mathbf{S}(\mathbf{v^0})p=\mathbf{B}\mathbf{E}(p',p)+\frac{1}{2t}\mathbf{B}\mathbf{\xg}:=\mathbf{L}(p',p)+\frac{1}{2t}\mathbf{B}\mathbf{\xg},\;\;\;\mathrm{in} \;(-\infty,-T_0].\label{ode2}
\ee
By (\ref{ode1}) we have that
\be
\mathbf{L}_k(h',h)=\sum_{i=1}^k(F_i(h'+\mathbf{e}\xz'+(\tilde{\rho}^0)',\tilde{h}+\mathbf{e}\xz+\tilde{\rho}^0)+\frac{n-1}{\xz}
+\frac{1}{2t}(\tilde{\rho}^0_i+h_i)),\label{elamou}
\ee
thus writing $p=(\overline{p},p_k)$ and $\mathbf{L}=(\overline{\mathbf{L}},L_k),$ the latter system decouples as
\bea\nonumber
\overline{p}'+\frac{1}{2t}\overline{p}+\xb D_{\overline{v}}\mathbf{\overline{S}}(\mathbf{\overline{v}^0})&=&\overline{\mathbf{L}}(p',p)+\frac{1}{2t}\mathbf{B}\mathbf{\xg},\;\;\;\mathrm{in} \;(-\infty,-T_0]\\
p_k'+\frac{1}{2t}p_k&=&L_k(p',p),\;\;\;\mathrm{in} \;(-\infty,-T_0].
\label{ode3}
\eea
Now, by (\ref{v}) we have
\bea\nonumber
D_{\overline{v}}\mathbf{\overline{S}}(\mathbf{\overline{v}^0})&=&-\sqrt{2}\mathbf{C} \left[ \begin{array}{ccccc}
e^{-\sqrt{2} v_1} & 0 &\cdots &0 \\
0 &e^{-\sqrt{2} v_2} & \cdots &0  \\
\vdots& &\ddots &\vdots  \\
0 &0 &\cdots & e^{-\sqrt{2}v_{k-1}}  \end{array} \right]\\ \nonumber
&=&\frac{e^{-\sqrt{2}\eta}}{2\xb}\mathbf{C}\left[ \begin{array}{ccccc}
a_1 & 0 &\cdots &0 \\
0 &a_2 & \cdots &0  \\
\vdots& &\ddots &\vdots  \\
0 &0 &\cdots & a_{k-1}  \end{array} \right],
\eea
where $a_l=(k-l)l,\;l=1,...,k-1,$
where the matrix $\mathbf{C}$ is given in (\ref{C}). $\mathbf{C}$ is symmetric and positive definite. Indeed, a straightforward computation yields that its eigenvalues are explicitly given by
$$1,\frac{1}{2},...,\frac{k-1}{k}.$$
We consider the symmetric, positive definite square root matrix of $\mathbf{C}$ and denote it by $\mathbf{C}^{\frac{1}{2}}.$ Then setting
$$\overline{p}=\mathbf{C}^\frac{1}{2}w,\quad \mathbf{Q}(w',p_k',w,p_k)=\mathbf{C}^{-\frac{1}{2}}\overline{\mathbf{L}}(\mathbf{C}^{\frac{1}{2}}w',p_k',\mathbf{C}^{\frac{1}{2}}w,p_k)$$
and
$$Q_k(w',p_k',w,p_k)=\overline{L_k}(\mathbf{C}^{\frac{1}{2}}w',p_k',\mathbf{C}^{\frac{1}{2}}w,p_k)$$
we see that equation (\ref{ode3}) becomes
\begin{align}\nonumber
w'+\frac{1}{2t}w+\frac{e^{-\sqrt{2}\eta(t)}}{2}\mathbf{A}w&=\mathbf{Q}(w',p_k',w,p_k)+\frac{1}{2t}\mathbf{C}^{-\frac{1}{2}}\mathbf{B}\mathbf{\xg},\\
p_k'+\frac{1}{2t}p_k&=Q_k(w',p_k',w,p_k)
\label{ode4}
\end{align}
where
$$\mathbf{A}=\mathbf{C}^\frac{1}{2}\left[ \begin{array}{ccccc}
a_1 & 0 &\cdots &0 \\
0 &a_2 & \cdots &0  \\
\vdots& &\ddots &\vdots  \\
0 &0 &\cdots & a_{k-1}  \end{array} \right]\mathbf{C}^\frac{1}{2}.$$
In particular $\mathbf{A}$ has positive eigenvalues $\xl_1,\xl_2,...,\xl_{k-1}.$ Let the orthogonal matrix $\mathbf{\xL}$ such that $\mathbf{D} = \mathbf{\xL}^T \mathbf{A}\mathbf{\xL},$ where $\mathbf{D}$  is the diagonal matrix such that $D_{ii}=\xl_i,\;i=1,...,k-1.$
Set now
$$\xo=\mathbf{\xL}^Tw,\qquad\mathbf{\xG}(\xo',p_k',\xo,p_k)=\mathbf{\xL}^TQ(\mathbf{\xL}\xo',\mathbf{\xL}\xo),$$
and
$$\xG_k(\xo',p_k',\xo,p_k)=Q_k(\mathbf{\xL}w',p_k',\mathbf{\xL}w,p_k)$$
we have that (\ref{ode4}) becomes equivalent to
\begin{align}\nonumber
\xo'+\frac{1}{2t}\xo+\frac{e^{-\sqrt{2}\eta(t)}}{2}\mathbf{D}\xo&=\mathbf{\xG}(\xo',p_k',\xo,p_k)+\frac{1}{2t}\mathbf{\xd},\quad\mathrm{in} (-\infty,-T_0]\\ \label{odepro}
p_k'+\frac{1}{2t}p_k&=\xG_k(\xo',p_k',\xo,p_k),\quad\mathrm{in} (-\infty,-T_0],
\end{align}
where $\mathbf{\xd}=\mathbf{\xL}^T\mathbf{C}^{-\frac{1}{2}}\mathbf{B}\mathbf{\xg}.$

We will solve (\ref{odepro}) by using the fix point Theorem in a suitable space with initial data $\xo(-T_0)=0$ and $p_k(-T_0)=0.$ If $(\xo,p_k)$ is a solution of the problem (\ref{odepro}) with initial data 0 then has the form
\begin{align}\nonumber
\xo_i(t)&=-\frac{1}{\sqrt{-t}g(t)}\int^{-T_0}_t \sqrt{-s}g(s)\left(\xG_i(\xo',p_k',\xo,p_k)+\frac{\xd_i}{2t}\right)ds,\quad i=1,...,k-1\\ \label{odefix}
p_k&=-\frac{1}{\sqrt{-t}}\int^{-T_0}_t \sqrt{-s}\xG_k(\xo',p_k',\xo,p_k)ds,
\end{align}
where $$g(t)=e^{\frac{1}{2}\int_{t}^{-\tilde{T}_0}e^{-\sqrt{2}\eta(s)}ds},$$
$T_0> \tilde{T}_0$ and $\tilde{T}_0$ has been defined in Lemma \ref{lemodesimpl1}.

By Lemma \ref{lemodesimpl1} we have

\be
\frac{1}{\sqrt{-t}g(t)}\int^{-T_0}_t \frac{\sqrt{-s}g(s)}{-s}ds\leq\frac{1}{g(t)}\int^{-T_0}_t \frac{g(s)}{-s}ds\leq C(\tilde{T}_0)\frac{1}{\log(T_0)}.\label{23}
\ee

Finally by Proposition \ref{remark2} and \eqref{elamou}, we have that there exists constant $C=C(n,k,\xs)>0$ such that

$$|\xG_i(0)|\leq\frac{C}{t}\;\;\forall i=1,...,k.$$
and
$$|\xG_k(h_1)-\xG_k(h_2)|\leq C\left(\frac{\log|t|}{|t|}\right)^{\frac{1}{2}+\frac{\xs}{\sqrt{2}}}||h_1-h_2||_{\xL}.$$

Let $A(\xo,p)$ be a solution of (\ref{odefix}), then  we have
\bea
|A_i(0)|\leq \frac{C_1}{\log(T_0)}\;\forall i=1,..k-1\quad \text{and}\quad|A_k(0)|\leq C_2.\label{fragma}
\eea
Similarly
\be
\frac{|t|}{\log|t|}|A_i(0)|\leq \frac{C_1}{\log(T_0)}\;\forall i=1,..k.\label{fragma2}
\ee
if we choose $T_0>1$ large enough.
We consider the space
$$X=\{(h,p)\in C^1(-\infty,-T_0]:\;||h||_\xL\leq\frac{4C_1}{\log(T_0)}\;\;\;and\;\;\;||p||_\xL\leq4C_2\},$$
where $C_1,\;C_2$ are the constants in (\ref{fragma}) and (\ref{fragma2}).

Now, we have

$$
|A_i(h_1,p_1)-A_i(h_2,p_2)|\leq \frac{C}{\log T_0}\left(||h_1-h_2||_\xL+||p_1-p_2||_\xL\right),\quad\forall i=1,...,k-1
$$
and for some $0<\xa<1$
$$
|A_k(h_1,p_1)-A_k(h_2,p_2)|\leq \frac{C}{T_0^\xa}\left(||h_1-h_2||_\xL+||p_1-p_2||_\xL\right).
$$
Also we have
$$
\frac{|t|}{\log|t|}|A_i'(h_1,p_1)-A_i'(h_2,p_2)|\leq \frac{C}{\log T_0}\left(||h_1-h_2||_\xL+||p_1-p_2||_\xL\right),\quad\forall i=1,...,k-1
$$
and
$$
\frac{|t|}{\log|t|}|A_k(h_1,p_1)-A_k(h_2,p_2)|\leq \frac{C}{T_0^\xa}\left(||h_1-h_2||_\xL+||p_1-p_2||_\xL\right).
$$

The result follows by Banach fixed point theorem, if we choose $T_0$ big enough. We observe that a posteriori, the equation satisfied by $h(t)$ 
yields that $h(t)\to 0$, with precise rate 
$$ 
|h(t)| \ \le \  \frac C {\log|t|}\quad \hbox{as } t\to -\infty .  
$$

\bigskip
\noindent\emph{Acknowledgment } This work has been supported by  Fondecyt grants  3140567 and 1150066, Fondo Basal CMM and by Millenium Nucleus CAPDE NC130017.

\end{document}